\documentclass{amsart}

\usepackage{amsmath,amsthm,amsfonts,amssymb,bm,graphicx,mathrsfs}
\usepackage{xcolor}
\definecolor{dark-red}{rgb}{0.4,0.15,0.15}
\definecolor{dark-blue}{rgb}{0.15,0.15,0.4}
\definecolor{medium-blue}{rgb}{0,0,0.5}
\usepackage[unicode]{hyperref}
\hypersetup{colorlinks, linkcolor={dark-red},citecolor={dark-blue},urlcolor={medium-blue},
pdftitle={Motohashi's fourth moment identity for non-archimedean test functions and applications},
pdfauthor={Valentin Blomer, Peter Humphries, Rizwanur Khan, and Micah Milinovich},
pdfnewwindow=true}

\makeatletter
\newcommand*{\defeq}{\mathrel{\rlap{%
			\raisebox{0.3ex}{$\m@th\cdot$}}%
		\raisebox{-0.3ex}{$\m@th\cdot$}}%
	=}
\makeatother
\DeclareMathOperator*{\Res}{Res}
\theoremstyle{plain}
\newtheorem{theorem}{Theorem}

\newtheorem{lemma}{Lemma}

\newtheorem{cor}[theorem]{Corollary}

\numberwithin{equation}{section}

\begin{document}

\author{Valentin Blomer}
\address{Mathematisches Institut, Bunsenstr. 3-5, 37073 G\"ottingen, Germany}
\email{\href{mailto:vblomer@math.uni-goettingen.de}{vblomer@math.uni-goettingen.de}}

\author{Peter Humphries}
\address{Department of Mathematics, University College London, Gower Street, London WC1E 6BT, UK}
\email{\href{mailto:pclhumphries@gmail.com}{pclhumphries@gmail.com}}

\author{Rizwanur Khan}
\address{Department of Mathematics, University of Mississippi, University, MS 38677, USA}
\email{\href{mailto:rrkhan@olemiss.edu}{rrkhan@olemiss.edu}}

\author{Micah B. Milinovich}
\address{Department of Mathematics, University of Mississippi, University, MS 38677, USA}
\email{\href{mailto:mbmilino@olemiss.edu}{mbmilino@olemiss.edu}}

\title[Motohashi's fourth moment identity for non-archimedean test functions]{Motohashi's fourth moment identity for non-archimedean test functions and applications}

\thanks{The first author is supported in part by DFG grant BL 915/2-2. The second author is supported by the European Research Council grant agreement 670239.}
 

\begin{abstract} Motohashi established an explicit identity between the fourth moment of the Riemann zeta function weighted by some test function and a spectral cubic moment of automorphic $L$-functions. By an entirely different method, we prove a generalization of this formula to a fourth moment of Dirichlet $L$-functions modulo $q$ weighted by a non-archimedean test function. This establishes a new reciprocity formula. As an application, we obtain sharp upper bounds for the fourth moment twisted by the square of a Dirichlet polynomial of length $q^{1/4}$. An auxiliary result of independent interest is a sharp upper bound for a certain sixth moment for automorphic $L$-functions, which we also  use to improve the best known subconvexity bounds for automorphic $L$-functions in the level aspect. 
 \end{abstract}


\setcounter{tocdepth}{2}  \maketitle 

\maketitle

 \section{Introduction} 

\subsection{A reciprocity formula} A landmark result in the theory of $L$-functions, both because of its structural beauty and its applications, is Motohashi's identity for the fourth moment of the Riemann zeta function \cite[Theorem 4.2]{Mo}: if $F$ is a sufficiently nice test function, then
\begin{equation}\label{fourth}
\int_{\mathbb{R}} |\zeta(1/2 + it)|^4 F(t) \, dt
\end{equation}
is equal to an explicit main term plus a cubic moment of the shape
\begin{equation}\label{fourth1}
\sum_j L(1/2, \psi_j)^3  \check{F}(t_j) + \text{ similar holomorphic and Eisenstein contribution},
\end{equation}
where the sum runs over Maa{\ss} forms $\psi_j$ with spectral parameter $t_j$ for the group $\mathrm{SL}_2(\mathbb{Z})$ and $\check{F}$ is a certain integral transform of $F$ given explicitly in terms of hypergeometric functions. Historically, this established the first reciprocity formula between two different families of $L$-functions. Choosing the test function $F$ appropriately, it can be used, for instance, to prove sharp upper bounds for the fourth moment of the Riemann zeta function on the critical line in short intervals $t \in [T, T+T^{2/3}]$. Motohashi's formula can also be inverted to some extent; Ivi\'c \cite{Iv} used this to obtain Weyl-type subconvexity bounds for the $L$-values $L(1/2, \psi_j)$. 

Motohashi's proof starts by opening the four zeta values as Dirichlet series and integrating over $t$, which, after a change of variables, gives a Dirichlet series containing a shifted convolution problem 
\begin{equation}\label{shifted}
 \tau(n)\tau(n+h).
\end{equation} 
  A spectral decomposition then yields the spectral cubic moment. 

A very different strategy was suggested by Michel and Venkatesh \cite[Section 4.5]{MV}: we interpret \eqref{fourth} as a second moment of $L$-functions associated with an Eisenstein series $E$ and choose $F$ as the corresponding local $L$-factors at infinity. Denoting the completed $L$-functions by $\Lambda(s, E)$, we have by Hecke's integral representation and Parseval's theorem (ignoring convergence)
\[\int_{\mathbb{R}} |\Lambda(1/2 + it, E)|^2 \, dt \approx \int_0^{\infty} |E(iy)|^2 \, dy.\]
Decomposing spectrally (and suppressing the continuous spectrum for notational simplicity), using Rankin--Selberg theory and Hecke's integral representation again, this ``equals''
\[\int_0^{\infty} \sum_{j} \langle |E|^2, \psi_j\rangle \psi(iy) \, dy \approx \int_0^{\infty} \sum_{j} \Lambda(1/2, \psi_j \times E) \psi(iy) \, dy \approx \sum_{j} \Lambda(1/2, \psi_j)^3.\]
This very beautiful idea comes with two technical challenges: (a) none of the integrals converge and some regularization is necessary, and (b) while this works very nicely for the special test function $F(t) = |L_{\infty}(1/2 + it , E)|^2$, it is not easy to spell out what happens for general test functions $F$. 

In this paper, we offer yet another proof of Motohashi's identity, which has the advantage of working nicely in greater generality. The set-up we are interested in is as follows. For $q \in \mathbb{N}$, consider
\[\sum_{\chi \hspace{-.25cm} \pmod{q}} F_{\text{fin}}(\chi) \int_{\mathbb{R}} |L(1/2 + it, \chi)|^4 F_{\infty}(t) \, dt.\]
This can be seen as the proper ad\`{e}lic analogue of \eqref{fourth}, twisting $\zeta(s)$ by  the complete family of $\mathrm{GL}(1)$ characters $n^{it}\chi(n)$.  By elementary Fourier analysis, every  $F_{\text{fin}}$ is a linear combination of character values. With applications in mind, we consider  test functions of the shape
\begin{equation}\label{test}
F_{\text{fin}}(\chi) = \chi(a) \overline{\chi}(b)
\end{equation}
for some integers $a, b \in \mathbb{N}$. It is straightforward to include a character average in Motohashi's proof, which  essentially  results in a shifted convolution problem \eqref{shifted} where $h$ is divisible by $q$. It is much less straightforward to include a general test function \eqref{test}, because then the shifted convolution problem becomes a sum over over $\tau(n)\tau(m)$ subject to the condition $an \equiv bm \pmod{q}.$ 
The difficulty of such an extension (with sufficient control on $a, b$) was already observed in \cite[p.\ 210]{DFI}. A heuristic argument based on a different strategy that we sketch in Subsection \ref{heur} suggests that we should expect something like
\begin{equation}\label{expect}
\sum_{\chi \hspace{-.25cm} \pmod{q}}\chi(a)\overline{\chi}(b) \int_{\mathbb{R}} |L(1/2 + it, \chi)|^4 F (t) \, dt \rightsquigarrow \frac{q^{1/2}}{a^{1/2}}\sum_{\text{level } ab}\lambda_j(q)\lambda_j(b)  L(1/2, \psi_j)^3 \check{F}(t_j). 
\end{equation}
This indicates that the period integral approach will not be straightforward to extend because at the very least some non-trivial combinatorics in the Hecke algebra (cf.\ \cite{Za} how this could look like in a slightly different situation) have to happen to generate the Hecke eigenvalues on the right-hand side.

We will present a  proof in the spirit of recent reciprocity formulae of the first and third authors \cite{BK2, BK1} that deal with this more general set-up without essential structural difficulties. We proceed to describe our first main result in detail. Let  $a, b, q \in \mathbb{N}$,  $s, u, v \in \mathbb{C}$, $F$ an even holomorphic function that is Schwartz class on fixed vertical lines, and  $f$ an automorphic form for $\mathrm{SL}_2(\mathbb{Z})$ that is either cuspidal or the standard Eisenstein series $\frac{d}{ds} E(z, s)|_{s = 1/2}$. We denote its Hecke eigenvalues by $\lambda(n)$, so that $\lambda(n) = \tau(n) \defeq \sum_{ab = n} 1$ if $f$ is Eisenstein.  We define
\begin{equation}\label{defT}
\mathcal{T}_{a, b, q}(s, u, v) \defeq \sum_{\substack{\chi \hspace{-.25cm} \pmod{q}\\ \chi \text{ primitive}}}\chi(a)\overline{\chi}(b) \int_{(0)} L(s +z, \chi)L(u +z, \chi) L(v-z, f \times \overline{\chi})  F(z) \, \frac{dz}{2\pi i}.
\end{equation}
where the integration is over the vertical line $\Re z = 0$. We may assume without loss of generality that $(a, b) = (ab, q) = 1$. It is convenient to assume that $F$ is divisible by $(1-u)(v-1)^2\prod_{j=1}^{50}(j-s)$. A typical function we have in mind is
\begin{equation}\label{defF}
F(z) = e^{z^2} (z^2-(1-u)^2)^2(z^2-(v-1)^2)^2\prod_{j=1}^{50}(z^2-(j-s)^2),
\end{equation}
which is positive for $\Re z = 0$, $s=u=v= 1/2$. 
To get a nice looking formula, we also need to include non-primitive characters, and for simplicity we assume that $q$ is prime. For a suitable correction polynomial $P_q(s, u, v, z) $  defined explicitly in \eqref{defP} below and satisfying
\begin{equation}\label{P}
  P_q(s, u, v, z) \ll 1 + |\lambda(q)|, \quad \Re s, \Re u, \Re v \geq 1/2, \quad \Re z = 0,
\end{equation}
we define the   analogue for the trivial character
\begin{equation}\label{defTtriv}
\mathcal{T}^{\textnormal{triv}}_{q}(s, u, v) \defeq \int_{(0)} \zeta(s +z)\zeta(u +z) L(v-z, f ) P_q(s, u, v, z) F(z) \, \frac{dz}{2\pi i}.
\end{equation}
Note that our assumptions on $F$ imply that the integrand is holomorphic and that we can shift the $z$-contour in any way we want. 

On the spectral side, we define 
\begin{equation}\label{defMaass}
\mathcal{M}^{\textnormal{Maa{\ss}}}_{a, b, q}(s, u, v) \defeq \sum_{A \mid ab} \sum_{\psi \in \mathcal{B}^{\ast}(A)} \Theta^{\textnormal{Maa{\ss}}}_{a, b, q}(s, u, v, \psi) \frac{L(\frac{s+u-1+2v}{2}, \psi) L( \frac{1-s+u}{2}, f \times \psi)}{L(1, \mathrm{Ad}^2 \psi)},
\end{equation}
where $\mathcal{B}^{\ast}(A)$ denotes an orthonormal Hecke basis of Maa{\ss} \emph{new}forms of level $A$ and $\Theta^{\textnormal{Maa{\ss}}}_{a, b, q}(s, u, v, \psi)$ is a (complicated, but) completely explicit expression defined in \eqref{defTheta} that satisfies
\begin{equation}\label{boundTheta}
\Theta^{\textnormal{Maa{\ss}}}_{a, b, q}(s, u, v, \psi) \ll_{s, u, v, F, \varepsilon} q^{1/2} A^{-1/2} (1 + |\lambda_{\psi}(q)|) (1 + |t_{\psi}|)^{-30} (abq)^{\varepsilon}
\end{equation}
for $\Re s, \Re u, \Re v = 1/2$  and $a\asymp b$. 
Similarly, we define
\[\mathcal{M}^{\textnormal{hol}}_{a, b, q}(s, u, v) \defeq \sum_{A \mid ab} \sum_{k \in 2\mathbb{N}}\sum_{\psi \in \mathcal{B}_k^{\ast}(A)} \Theta^{\textnormal{hol}}_{a, b, q}(s, u, v, \psi) \frac{L(\frac{s+u-1+2v}{2}, \psi) L( \frac{1-s+u}{2}, f \times \psi)}{L(1, \mathrm{Ad}^2 \psi)},\]
where $\mathcal{B}_k^{\ast}(A)$ denotes an orthonormal Hecke basis of holomorphic  newforms of weight $k$ and level $A$ and $\Theta^{\textnormal{hol}}_{a, b, q}(s, u, v, \psi)$   satisfies the analogous bound 
\begin{equation}\label{thetahol}
\Theta^{\textnormal{hol}}_{a, b, q}(s, u, v, \psi) \ll_{s, u, v, F, \varepsilon}q^{1/2} A^{-1/2} k^{-30} (abq)^{\varepsilon}
\end{equation}
for $\Re s, \Re u, \Re v = 1/2$  and $a\asymp b$. For simplicity, we assume that $a, b$ are squarefree, so that the Eisenstein spectrum is parametrized by $\tau(ab)$ cusps. We   define (initially in $\Re (s+u+2v) > 3$ and $ \Re (u-s) > 1$) 
\begin{displaymath}
\begin{split}
&\mathcal{M}^{\textnormal{Eis}}_{a, b, q}(s, u, v)\\
& \defeq  \int_{\mathbb{R}}  \Theta^{\textnormal{Eis}}_{a, b, q}(s, u, v,t)   \frac{\zeta(\frac{s+u-1+2v}{2} + it)\zeta(\frac{s+u-1+2v}{2} - it) L( \frac{1-s+u}{2} + it, f )L( \frac{1-s+u}{2} - it, f)}{\zeta(1+ 2it) \zeta(1 - 2it)} \, \frac{dt}{2\pi}
\end{split}
\end{displaymath}
where $\Theta^{\textnormal{Eis}}_{a, b, q}(s, u, v, t)$ is defined in \eqref{defthetaeis} and satisfies
\begin{equation}\label{boundthetaeis}
\Theta^{\textnormal{Eis}}_{a, b, q}(s, u, v, t) \ll_{s, u, v, \varepsilon, F}  (abq)^{\varepsilon}  q^{1/2}(ab)^{ \theta-1/2} (1+|t|)^{-30}
\end{equation}
for $\Re s = \Re v = \Re u = 1/2$, $t\in \mathbb{R}$, where $\theta \leq 7/64$ is an admissible exponent for the Ramanujan conjecture for the fixed form $f$ (in particular, $\theta = 0$ if $f$ is holomorphic or Eisenstein). 
While all three expressions $\mathcal{M}^{\ast}_{a, b, q}(s, u, v)$ for $\ast \in \{\textnormal{Maa{\ss}}, \textnormal{hol}, \textnormal{Eis}\}$ are holomorphic in $\Re s, \Re u, \Re v \geq 1/2$, the meromorphic continuation of $\mathcal{M}^{\textnormal{Eis}}_{a, b, q}(s, u, v)$ to $\Re (s+u+2v) < 3, \Re (u-s) < 1$ involves an  additional polar term, defined in \eqref{P3}. We define 
\[\mathcal{M}_{a, b, q}(s, u, v) \defeq \mathcal{M}^{\textnormal{Maa{\ss}}}_{a, b, q}(s, u, v)+ \mathcal{M}^{\textnormal{hol}}_{a, b, q}(s, u, v) + \mathcal{M}^{\textnormal{Eis}}_{a, b, q}(s, u, v).\]
We are now ready to state the reciprocity formula to which we have already alluded.

\begin{theorem}\label{thm1} Let $q, a, b \in \mathbb{N}$, $q$ prime, $(ab, q) = (a, b) = 1$, $a, b$ squarefree, $a \asymp b$.  Let $1/2 \leq \Re s, \Re u, \Re v < 3/4$ and $\Re s \leq \Re u$. Suppose that $F$ is holomorphic, Schwartz class on vertical lines, and divisible by  $ (1-u)(v-1)^2\prod_{j=1}^{50}(j-s)$. Then
\begin{equation}\label{formula}
\mathcal{T}_{a, b, q}(s, u, v)  + \mathcal{T}^{\textnormal{triv}}_{q}(s, u, v) = \mathcal{P}_{a, b, q}(s, u, v) + \mathcal{M}_{a, b, q}(s, u, v),
\end{equation}
where the ``main  term'' $\mathcal{P}_{a, b, q}(s, u, v) $ is defined in \eqref{Pall} and satisfies
\begin{equation}\label{boundP}
\mathcal{P}_{a, b, q}(s, u, v) \ll_{s, u, v, \varepsilon, F} q(ab)^{-1/2+\theta} (abq)^{\varepsilon}
\end{equation}
for $\Re s= \Re u = \Re v = 1/2$, $a\asymp b$, where $\theta$ is an admissible exponent for the Ramanujan conjecture for $f$. 
\end{theorem}

We emphasize that even though $\mathcal{M}_{a, b, q}(s, u, v)$ depends on $q$, it only involves the spectrum of level $ab$. This is the ultimate  reason for the specific design of the term $\mathcal{T}^{\textnormal{triv}}_{q}(s, u, v)$. In this sense, our formula is a clean reciprocity formula, where the pair (level, arithmetic  of weight function) on the Dirichlet side is $(q, ab)$ and on the spectral side is $(ab, q)$. 

Theorem \ref{thm1} contains a number of simplifying assumptions, most of which can be removed without any  structural difficulties at the cost of more technical work. If $q$ is not prime, we need slightly more complicated correction terms for non-primitive characters. The assumption that $F$   has zeros at $ 1-u, v-1, j-s$, $1 \leq j \leq 50$,  can be relaxed considerably, and probably entirely removed, but it saves us from computing several polar terms and buys us convergence without any trickery. The regularity assumptions on $F$ can also be relaxed. The assumption that $a$ is squarefree is only to keep the formulae a little simpler. The assumption that $b$ is squarefree is slightly more serious and enables us to use the Kuznetsov formula in a version that  involves only Fourier expansions at infinity.   For arbitrary $b$, one can use the analysis of K{\i}ral--Young \cite[Lemma 2.5 \& Theorem 3.4]{KY2} instead. As mentioned before, the assumption $(ab, q)= (a, b)= 1$ is without loss of generality, and if $a$ and $b$ are not of the same order of magnitude, our bounds may deteriorate by $(\max(b, a)/\min(b, a))^{O(1)}$ (this is unavoidable; cf.\ the sketch in Section \ref{heur}). 

The spectral side \eqref{fourth1} of Motohashi's original formula goes deeper in the spectrum (i.e.\ the support of $\check{F}$ is larger) the more complicated the test function $F$ is (e.g.\ in terms of oscillation). Our formula features a similar phenomenon for the non-archimedean test function, except that the spectral support now increases, in some sense  orthogonally,  in terms of the level instead of the spectral parameter.  

\subsection{A sixth moment}
In practice, we want to estimate the right-hand side of \eqref{expect} for large $q$ and somewhat large $a, b$, and a possible problem could be the occurrence of $\lambda_j(q)$ in \eqref{boundTheta} for Maa{\ss} forms $\psi_j$ for which the Ramanujan conjecture is not known. The factor $\lambda_j(b)$ is not a problem, since $b$ divides the level; cf.\ \eqref{stein}. A trivial bound on $\lambda_j(q)$, however, may invoke an undesirable factor of $q^{\theta}$ due to our limited knowledge of the Ramanujan conjecture.    In order to avoid this, one may try to use the extra average over the forms of level $ab$ and apply the Cauchy--Schwarz inequality. This is successful if there is an additional average over $a, b$, and to this end we will prove the following sixth moment bound, which is of independent interest.
\begin{theorem}\label{thm2} Let $Q, T \geq 1$, and for $q \in \mathbb{N}$, let $\mathcal{B}^{\ast}(q)$ denote an orthonormal basis of Hecke-Maa{\ss}  newforms $\psi$ of level $q$ having spectral parameter $t_{\psi}$.  Then 
\[\sum_{q \leq Q} \sum_{\substack{\psi \in \mathcal{B}^{\ast}(q)\\ |t_\psi| \leq T}} \frac{|L(1/2, \psi)|^6 }{L(1, \mathrm{Ad}^2 \psi)}  \ll_{\varepsilon} (QT)^{\varepsilon} T^{8} Q^{2}.\]
\end{theorem}
The emphasis here is on the $Q$-aspect, which is sharp up to the presence of $Q^{\varepsilon}$; the $T$-aspect only needs to be polynomial. 
For comparison, it is classical, although technically difficult, to understand the fourth moment for an individual large level $q$. Our result is easier because we have an additional average over $q$ (which, however, is spectrally not easy to exploit), but also harder because we study a higher moment. Any spectral method will have to complete the discrete spectral sum to an entire spectral expression including Eisenstein series, and already in the fourth moment one of the biggest obstacles is the fact that the additional continuous contribution is quite large in the level aspect. It is not surprising that this becomes even worse for the sixth moment, and here the artificially added Eisenstein term exceeds the targeted bound by a substantial power of $Q$.

An overview of the method of proof and how the various technical and conceptual issues are addressed will be given in Section \ref{heur}.  We present an immediate application of Theorem \ref{thm2}. 
\begin{cor}\label{cor3} Let $q$ be prime. Then
\[\sum_{\psi \in \mathcal{B}^{\ast}(q)} L(1/2, \psi)^5 e^{-t_\psi^2} \ll_{\varepsilon}  q^{1+\theta/3 + \varepsilon}\]
for every $\varepsilon > 0$. In particular, for $\psi \in \mathcal{B}^{\ast}(q)$, we have
\[L(1/2, \psi) \ll_{t_\psi, \varepsilon} q^{ \frac{1}{5} + \frac{\theta}{15} + \varepsilon} \ll q^{0.2073}.\]
\end{cor}
This improves the $\theta$-dependence of the fifth moment bound in \cite[Theorem 3]{BK1} and provides the numerical  subconvexity record for $L(1/2, \psi)$ in the level aspect (the previous exponent being $0.217$ from \cite[Theorem 4]{BK1}; cf.\ \cite{KY1}). 

\subsection{Fourth moments twisted by Dirichlet polynomials} For many applications, in particular  with respect to the amplification,  mollification, or resonance method, one wishes to augment moment results on $L$-functions by inserting well-chosen Dirichlet polynomials --- ideally as long as possible --- that in effect often act as additional fractional moments. This is classical for the Riemann zeta function, where Watt \cite{Wa}, building on work of Deshouillers and Iwaniec \cite{DI}, proved
\[\int_0^T  |\zeta(1/2 + it)|^4 \Bigl|\sum_{m \leq M} a_m m^{it}\Bigr|^2 \, dt \ll \| \textbf{a} \|_{\infty} (MT)^{1+\varepsilon}\]
for $M \leq T^{1/4}$ and an arbitrary sequence $(a_m)_{1 \leq m \leq M}$. This can be turned into an asymptotic formula; see \cite{Mo2, HY, BBLR}. Versions for Dirichlet $L$-functions with conductors sufficiently small with respect to $T$ can be found in \cite{HWW}, along with applications to primes in arithmetic progressions and short intervals. 
 
As an application of Theorems \ref{thm1} and \ref{thm2}, we will prove the following analogous sharp upper bound for a fourth moment of Dirichlet $L$-functions twisted by the square of a Dirichlet polynomial of length up to $q^{1/4}$. 
\begin{theorem}\label{thm4} Let $q$ be a prime, $1 \leq M \leq q^{1/4}$,  and $\{a(m)\}_{1 \leq m \leq M}$ a sequence of complex numbers supported on squarefree numbers. Then
\[\sum_{\chi \hspace{-.25cm} \pmod{q}} \Bigl|\sum_{m \leq M} a(m) \chi(m) \Bigr|^2 |L(1/2, \chi)|^4 \ll_{\varepsilon} \| \textbf{a} \|_{\infty} (Mq)^{1+\varepsilon} .\]
\end{theorem}
To get a feeling for the strength of the result, we mention that it implies trivially the Burgess bound $L(1/2, \chi) \ll_{\varepsilon} q^{3/16+\varepsilon}$ for every non-trivial character modulo $q$. The reader may wonder to what extent this upper bound can be turned into an asymptotic formula, but interestingly this is a much harder problem than in the case of the Riemann zeta function. The reason is that a $\delta$-mass at the point 1/2 is not a proper test function. On a technical level, the $t$-integral with a holomorphic test function and the freedom to shift its contour is crucial to establish convergence throughout the argument. Therefore, a corresponding asymptotic formula can be achieved if an additional $t$-average (essentially of constant length) is included, but for the central point individually, one has to use other methods (see e.g.\  \cite{Ho}, \cite{Za2}) that yield much weaker results. 
 
\subsection{Heuristics}\label{heur} We conclude the introduction with a heuristic argument supporting the relation \eqref{expect} and the bound in Theorem 2 as well as some additional comments. This section is not intended to provide proofs, but may serve as  a roadmap. 

We start with \eqref{expect}. For the sake of argument, we will use approximate functional equations, although our proof works with Dirichlet series in the region of absolute convergence and continues meromorphically only at the very end (the great advantage of this is that we do not have to deal with a root number term, and so we will ignore this term also in the present sketch). For simplicity, we will also ignore the $t$-average whose purpose is to achieve convergence, as well as all ``main terms'' that arise in the course of the computation. We have
\begin{displaymath}
\begin{split}
\sum_{\chi \hspace{-.25cm} \pmod{q}} \chi(a)\overline{\chi}(b) |L(1/2, \chi)|^4 &\approx \sum_{\chi \hspace{-.25cm} \pmod{q}} \chi(a)\overline{\chi}(b)\sum_{n, m, r_1, r_2 \asymp q^{1/2}} \frac{\chi(nm) \overline{\chi}(r_1r_2)}{(nmr_1r_2)^{1/2}} \approx \sum_{\substack{n, m, r_1, r_2 \asymp q^{1/2}\\ anm \equiv br_1r_2 \hspace{-.25cm} \pmod{q}}} 1.
\end{split}
\end{displaymath}
Rather than solving a shifted convolution problem, we take an asymmetric approach and apply Poisson summation only in one variable, say $n $. This gives
\[\frac{1}{q}\sum_{n, m, r_1, r_2 \asymp q^{1/2}} e\left(\frac{\overline{a}bn\overline{m} r_1r_2}{q}\right).\]
Suppose that $a \asymp b$. Then $bnr_1r_2 \asymp  amq$, so we can apply the additive reciprocity formula
\begin{equation}\label{addrec}
e\left(\frac{n \overline{d}}{c}\right) = e\left(-\frac{n \overline{c}}{d}\right) e\left(\frac{n}{cd}\right)
\end{equation}
to obtain
\[\frac{1}{q}\sum_{n, m, r_1, r_2 \asymp q^{1/2}} e\left(\frac{\overline{q}bnr_1r_2}{am}\right).\]
Applying Poisson summation in $n, r_1, r_2$, this gives roughly
\[\frac{1}{a^2} \sum_{m \asymp q^{1/2}} \sum_{n, r_1, r_2 \asymp a} S(qr_1r_2\overline{b}, n, am).\]
If we assume for simplicity that $b$ is prime and coprime to $aqmr_1r_2$ (this is where the assumption ``$b$ squarefree'' in Theorem \ref{thm1} is used), then 
$ S(qr_1r_2 , bn, abm) = -  S(qr_1r_2\overline{b}, n, am)$ by twisted multiplicativity. For the Kloosterman sum on the left-hand side, we are in the ``Linnik range'' $\sqrt{qr_1r_2 bn} \asymp abm$, and an application of the Kuznetsov formula yields the right-hand side of \eqref{expect}. 

A back-of-the-envelope computation for Theorem \ref{thm2} looks as follows. By an approximate functional equation, we have roughly
\[L(1/2, \psi)^6 \approx \sum_{n, m \ll Q^{3/2}} \frac{\tau_3(n)\tau_3(m) \lambda_{\psi}(n)\lambda_{\psi}(m)}{(nm)^{1/2}}\]
for $\psi \in \mathcal{B}^{\ast}(q)$, $q \asymp Q$, where for simplicity we regard $T$ as fixed; here $\tau_3(n) \defeq \sum_{abc = n} 1$. Summing $\psi \in \mathcal{B}^{\ast}(q)$ and $q \asymp Q$ by the Kuznetsov formula, the diagonal term is of size $Q^2$ and the off-diagonal term looks roughly like
\begin{equation}\label{kuz-sketch}
Q \sum_{q \asymp Q} \sum_{n, m \ll Q^{3/2}} \frac{\tau_3(n)\tau_3(m)  }{(nm)^{1/2}}\sum_{c \ll Q^{1/2}} \frac{S(n, m, qc)}{qc}.
\end{equation}
The key idea is to switch the roles of $q$ and $c$ and to apply the Kuznetsov formula  backwards, but this time viewed as a spectral summation formula of level $c$. This switching principle is well-known from sieve theory; here we apply it in an automorphic context. We obtain roughly
\[Q^{1/2} \sum_{c \ll Q^{1/2}} \sum_{\psi \in \mathcal{B}^{\ast}(c)}\sum_{n, m \ll Q^{3/2}} \frac{\tau_3(n)\tau_3(m) \lambda_{\psi}(n)\lambda_{\psi}(m)}{(nm)^{1/2}}.\]
Applying Vorono\u{\i} summation on the long $n, m$-sum, we may hope to get complete square root cancellation, obtaining the final bound $Q^{3/2}$ for the off-diagonal contribution.

 Apart from neglecting oldforms, whose presence is technically challenging, 
this heuristic argument  has an important deficiency: it ignores the continuous spectrum that needs to be added artificially before applying the Kuznetsov formula, and this contribution is of size $Q^{5/2}$ and  exceeds substantially our target bound. In particular, it is impossible to estimate \eqref{kuz-sketch} by $Q^{3/2}$ as indicated, as we know in advance that it is of size $Q^{5/2}$. This dilemma of a gigantic continuous spectrum contribution is well-known to experts and was first encountered in \cite{DFI2}, where the contribution was carefully computed and matched with another main term that occurred at a different stage of the argument. In \cite{BHM}, the problem was solved by introducing additional zeros in the Mellin transform of the weight function in the approximate functional equation. Unfortunately, this loses positivity (and therefore many convenient simplifications), and it is also a very technical task to find the initial zeros at the end of the argument where they are needed to make a certain main term disappear. In the present situation, we argue differently and find a rather soft way to match two Eisenstein terms without actually computing them. 

\section{Preliminaries}
\subsection{Hecke theory} We generally denote Hecke eigenvalues, with or without subscript, by $\lambda(n)$. For newforms of level $N$, we will often use the multiplicativity relation
\begin{equation}\label{mult}
\lambda(nm) = \sum_{\substack{d\mid (m, n)\\ (d, N) = 1}} \mu(d) \lambda(m/d) \lambda(n/d).
\end{equation}
We have the general upper bound 
\begin{equation}\label{hecke0} 
 \lambda(n) \ll_{\varepsilon} n^{\theta + \varepsilon}.
\end{equation}
For a newform of level $N = N_1N_2$ with $N_1$ squarefree, $(N_1, N_2) = 1$, and some $n \mid N_1$, we have
\begin{equation}\label{stein}
 |\lambda(n)| = n^{-1/2}
\end{equation}
via \cite[Theorem 2]{Og}.
\subsection{Functional equation for the Hurwitz zeta function} For $\alpha \in \mathbb{R}$, $\Re s > 1$, let
\[\zeta(s, \alpha) \defeq \sum_{n + \alpha > 0} (n+\alpha)^{-s}\]
denote the Hurwitz zeta function. It has meromorphic continuation to all $s \in \mathbb{C}$ with a simple pole at $s=1$ of residue 1 and satisfies the functional equation
\begin{equation}\label{hur}
\zeta(s, \alpha) = \sum_{\pm}G^{\mp}(1-s) \zeta^{(\pm \alpha)}(1-s), 
\end{equation}
where
\[G^{\pm}(s) = (2\pi)^{-s} \Gamma(s)\exp(\pm i\pi s/2)\]
and 
$\zeta^{(\alpha)}(s)$ is (the meromorphic continuation of) $\sum_n e(\alpha n) n^{-s}$. For $\alpha \in \mathbb{Q}$, this is a reformulation of Poisson summation in residue classes.

\subsection{Functional equation for twisted automorphic \texorpdfstring{$L$}{L}-functions} For $\alpha \in \mathbb{R}$, $\Re s > 1$, let
\begin{equation}\label{twistedL}
L(s, \alpha, f) \defeq \sum_n  \lambda_f(n) e(\alpha n)n^{-s},
\end{equation}
where, as before, $f$ is a  Hecke eigenform of the group $\mathrm{SL}_2(\mathbb{Z})$, either Maa{\ss} with spectral parameter $t$ and parity $\epsilon \in \{\pm 1\}$, or holomorphic of weight $k$,  or the standard Eisenstein series with $\lambda_f(n) = \tau(n)$. If $\alpha = a/c \in \mathbb{Q}$ with $(a, c) = 1$, this $L$-function has meromorphic continuation to all $s \in \mathbb{C}$ with a double pole at $s=1$ with Laurent expansion
\begin{equation}\label{laurent}
\frac{1}{c}\left(\frac{1}{(s-1)^2} + \frac{2\gamma - 2\log c}{s-1} + O(1)\right)
\end{equation}
if $f$ is Eisenstein; note that this is independent of $a$. The twisted $L$-function satisfies the functional equation (see e.g.\ \cite[Section 2.4]{HM})
\begin{equation}\label{vor}
L(s, a/c, f) = \sum_{\pm} G^{\mp}_f(1-s) c^{1-2s} L(s, \pm \overline{a}/c, f),
\end{equation}
where
\[G^{+}_f(s) = i^k (2\pi)^{1-2s} \frac{\Gamma(s + \frac{k-1}{2})}{\Gamma(1-s + \frac{k-1}{2})}, \quad G^-_f(s) = 0\]
if $f$ is holomorphic of weight $k$ and
\begin{equation}\label{Gf}
G^{\pm}_f(s) = \epsilon^{(1\mp 1)/2}\frac{\Gamma(\frac{1}{2}(s + it))\Gamma(\frac{1}{2}(s - it))}{ \Gamma(\frac{1}{2}(1-s + it))\Gamma(\frac{1}{2}(1-s - it))} \mp \frac{\Gamma(\frac{1}{2}(1+s + it))\Gamma(\frac{1}{2}(1+s - it))}{ \Gamma(\frac{1}{2}(2-s + it))\Gamma(\frac{1}{2}(2-s - it))}
\end{equation}
if $f$ is Maa{\ss} with with spectral parameter $t$ and parity $\epsilon \in \{\pm 1\}$. This also holds for $f$ equal to the standard Eisenstein series with $t = 0$ and $\epsilon = 1$.  

\subsection{Fourier coefficients}\label{kuznetsov} We quote from \cite[Section 3]{BK1} and refer to this source for more details and references. 
The cuspidal spectrum is parametrized by  pairs $(\psi, M)$ of $\Gamma_0(N)$-normalized newforms $\psi$ of level $N_0 \mid N$ and integers $ M \mid N/N_0$. The corresponding Fourier coefficients are
 \begin{equation}\label{rho-cusp}
 \rho_{\psi, M, N}(n) = \frac{1}{L(1, \mathrm{Ad}^2 \psi)^{1/2}(N\nu(N))^{1/2}} \prod_{p \mid N_0} \left(1 - \frac{1}{p^2}\right)^{1/2}  \sum_{d \mid (M,n)} \xi_{\psi}(M, d) \frac{d}{M^{1/2}} \lambda_\psi(n/d)  
 \end{equation}
 for $n \in \mathbb{N}$, where $\nu(N) = \prod_{p \mid N} (1 + 1/p)$ and the multiplicative function $\xi_{\psi}$ is  defined in \cite[(3.10)]{BK1} and satisfies in particular
\begin{equation}\label{xi-explicit}
\xi_{\psi}(p, p) = \left(1 - \frac{\lambda_{\psi}(p)^2}{p(1 + 1/p)^2}\right)^{-1/2}, \quad \xi_{\psi}(p, 1) = \frac{-\lambda_{\psi}(p)}{1 + 1/p^2}, \quad \xi_{\psi}(1,1) = 1
\end{equation}
for $p \nmid N_0$ and in general  
\begin{equation}\label{xi-arithmetic}
\xi_{\psi}(M, d)   \ll_{\varepsilon} M^{\varepsilon}(M/d)^{\theta}. 
\end{equation}
For $-n \in \mathbb{N}$ we have $\rho_{\psi, M, N}(n) = \epsilon_\psi \rho_{\psi, M, N}(-n)$ if $\psi$ is Maa{\ss} of parity $\epsilon_\psi \in \{\pm 1\}$ and $\rho_{\psi, M, N}(n) = 0$ if $\psi$ is holomorphic. 

If $N$ is squarefree, the Fourier coefficients of Eisenstein series of level $N$ are easy to describe. They are parametrized by divisors $v\mid N$  and a continuous parameter $s = 1/2 + it$. The corresponding Fourier coefficients are given by  (see e.g.\ \cite[(3.25)]{CI})
  \begin{equation}\label{rho-eis}
  \begin{split}
    \rho_{v,  N}(n, t)  = &\frac{C(v, M, t)}{(Nv)^{1/2} \zeta^{(N)}(1 + 2it)}   \sum_{b \mid v} \sum_{\gamma \mid N/v} \mu(b\gamma) b \left(\frac{b}{\gamma}\right)^{it} \eta\left(\frac{|n|}{b\gamma}, t\right), 
       \end{split}
    \end{equation}
where $\eta(n, t) =    \sum_{d_1d_2 = n} (d_1/d_2)^{it}$ for $n \in \mathbb{N}$ (and 0 otherwise) and  $|C(v, N, t)| = 1$. 
For general $N$, we follow \cite[Section 3]{BK1} and parametrize unitary Eisenstein series of $\Gamma_0(N)$ by a continuous parameter  $s = 1/2 + it$  together with pairs $(\chi, M)$, where $\chi$ is a primitive Dirichlet character of conductor $c_{\chi}$ and $M \in \mathbb{N}$ satisfies $c_{\chi}^2 \mid M \mid N$. We write 
\begin{equation*}
\begin{split}
&  \tilde{\mathfrak{n}}_N(M) = \Bigl(\prod_{\substack{p \mid N\\ p \nmid (M, N/M)}} \frac{p}{(p+1)} \prod_{ p \mid (M, N/M)} \frac{p-1}{p+1}\Bigr)^{1/2},\\
&M = c_{\chi} M_1 M_2, \quad \text{where} \quad (M_2, c_{\chi}) = 1, \quad M_1 \mid c_{\chi}^{\infty},
\end{split}
\end{equation*}
 so that $c_{\chi} \mid M_1$ and $(M_1, M_2) = 1$. The Fourier coefficients of the Eisenstein series attached to the data $(N, M, t, \chi)$ are
  \begin{equation}\label{rho-eis1}
  \begin{split}
    \rho_{\chi, M, N}(n, t)  = &\frac{\tilde{C}(\chi, M, t)|n|^{it}}{(N \nu(N))^{1/2}\tilde{ \mathfrak{n}}_N(M) L^{(N)}(1 + 2it, \chi^2)}  \left(\frac{M_1}{M_2}\right)^{1/2}   \sum_{\delta \mid M_2} \delta \mu\Big(\frac{M_2}{\delta}\Big)  \overline{\chi}(\delta)  \sum_{\substack{cM_1\delta f =  n \\ (c, N/M) = 1}}\frac{\chi(c)}{c^{2it}}   \overline{\chi}(f) ,
    \end{split}
    \end{equation}
 where $|\tilde{C}(\chi, M, t)| = 1$.

\subsection{The Kuznetsov formula} For $x > 0$, we define the integral kernels 
\begin{displaymath}
\begin{split}
& \mathcal{J}^+(x, t) \defeq \frac{\pi  i}{ \sinh(\pi t)} (J_{2 it}(4 \pi x) - J_{-2it}(4 \pi x)), \\
& \mathcal{J}^-(x, t) \defeq \frac{\pi  i}{ \sinh(\pi t)} (I_{2 it}(4 \pi x) - I_{-2it}(4 \pi x)) = 4 \cosh(\pi t) K_{2it}(4\pi x),\\
& \mathcal{J}^{\textnormal{hol}}(x, k) \defeq 2\pi i^k J_{k-1}(4\pi x) =  \mathcal{J}^+(x, (k-1)/(2i)) , \quad k \in 2\mathbb{N}. 
 \end{split}
 \end{displaymath}
 If $H \in C^3((0, \infty))$ satisfies  $x^j H^{(j)}(x) \ll \min(x, x^{-3/2})$ for $0 \leq j \leq 3$, we define
\[\mathscr{L}^{\diamondsuit}H = \int_{0}^{\infty} \mathcal{J}^{\diamondsuit}(x, .) H(x) \, \frac{dx}{x}\]
 for $\diamondsuit \in \{+, -, \textnormal{hol}\}$, and for  $n, m, N \in \mathbb{N}$,  we have  
 \begin{equation}\label{kuz2} 
\begin{split}
\sum_{N \mid c} &\frac{S(\pm n, m, c)}{c}H \left(\frac{\sqrt{nm}}{c}\right) \\
&= \mathcal{A}_N^{\textnormal{Maa{\ss}}}(\pm n, m; \mathscr{L}^{\pm}H) +  \mathcal{A}_N^{\textnormal{Eis}}(\pm n, m; \mathscr{L}^{\pm}H)+   \mathcal{A}_N^{\textnormal{hol}}(\pm n, m; \mathscr{L}^{\textnormal{hol}}H),    \end{split}
\end{equation}
where
  \begin{equation}\label{aa1}
\begin{split}
& \mathcal{A}_N^{\textnormal{Maa{\ss}}}(n, m; h) \defeq \sum_{N_0M \mid N}   \sum_{\psi \in \mathcal{B}^{\ast}(N_0)} \rho_{\psi, M, N}(n) \overline{ \rho_{\psi, M, N}(m)} h(t_\psi),\\
&  \mathcal{A}_N^{\textnormal{Eis}}(n, m; h) \defeq \sum_{v \mid N} \int_{\mathbb{R}} \rho_{v, N}(n, t) \overline{ \rho_{v, N}(m, t)} h(t) \, \frac{dt}{2\pi} \quad (N \text{ squarefree}),\\
&  \mathcal{A}_N^{\textnormal{Eis}}(n, m; h) \defeq \sum_{c_{\chi}^2\mid M \mid N} \int_{\mathbb{R}} \rho_{\chi, M, N}(n, t) \overline{ \rho_{\chi, M, N}(m, t)} h(t) \, \frac{dt}{2\pi} \quad  (\text{in general}),\\
& 
\mathcal{A}_N^{\textnormal{hol}}(n, m; h) \defeq  \sum_{N_0M \mid N}   \sum_{ \psi \in \mathcal{B}_{\textnormal{hol}}^{\ast}(N_0) } \rho_{\psi, M, N}(n) \overline{ \rho_{\psi, M, N}(m)} h(k_\psi).
 \end{split}
\end{equation}
(Recall that $\psi \in \mathcal{B}^{\ast}(N_0)$ and $\psi \in \mathcal{B}_{\textnormal{hol}}^{\ast}(N_0)$ are $\Gamma_0(N)$-normalized.) 
Conversely, if $h$ is holomorphic in an $\varepsilon$-neighbourhood of $|\Im t| \leq 1/2$ and satisfies $h(t) \ll (1 + |t|)^{-2-\delta}$ in this region for some $\delta > 0$, then for $n, m \in \mathbb{N}$, we have \cite[(3.14)]{BK1}
\begin{equation}\label{kuz-con}
 \mathcal{A}_N^{\textnormal{Maa{\ss}}}(n, m; h) +  \mathcal{A}_N^{\textnormal{Eis}}(n, m; h) = \delta_{n, m} \int_{-\infty}^{\infty} h(t) \frac{t \tanh(\pi t) \, dt}{2\pi^2} + \sum_{N \mid c} \frac{S(n, m, c)}{c} \mathscr{K} h\left(\frac{\sqrt{nm}}{c}\right),
\end{equation}
where
\begin{equation}\label{trafo-K}
\mathscr{K}h(x) = \int_{-\infty}^{\infty} \mathcal{J}^+(x, t) h(t) t \tanh(\pi t) \, \frac{dt}{2\pi^2}= \frac{i}{\pi} \int_{-\infty}^{\infty} \frac{ J_{2it}(4\pi x)}{\cosh(\pi t)} h(t)  t   \, dt 
\end{equation}

\subsection{Integral transforms}\label{26} We generalize \eqref{trafo-K} slightly and define for $s\in \mathbb{C}$ the transform $\mathscr{K}_sh$ by 
$\mathscr{K}_sh(x) \defeq x^s \mathscr{K}h(x).$ 

\begin{lemma}\label{simple-K} Let $s \in \mathbb{C}$ with $\Re s < -10$, and suppose that $h$ is holomorphic in $|\Im t | < (-\Re s + 15)/2$, satisfying $h(t) \ll (1+|t|)^{-10}$ and having zeros at $\pm i(2n-1)/2$, $n \in \mathbb{N}$, in this region. Then $H \defeq \mathscr{K}_sh$ satisfies the assumptions of \eqref{kuz2}, i.e. $x^j H^{(j)}(x) \ll_s \min(x, x^{-3/2})$ for $0 \leq j \leq 3$. 
\end{lemma} 

\begin{proof}
We record the formula \cite[(8.411.10)]{GR} 
\begin{equation}\label{J1}
\frac{J_{2 i t}(4\pi x)}{\cosh(\pi t)} = \frac{(2\pi x)^{2it}}{\sqrt{\pi}\Gamma(1/2 + 2 i t)\cosh(\pi t)} \int_{-1}^1 (1-y^2)^{2it - 1/2} \cos(4\pi x y) \, dy
\end{equation}
for $\Re(2 i t) > -1/2$, $x > 0$. In particular, 
\begin{equation}\label{inparticular}
\frac{d^j}{dx^j} \frac{J_{2it}(4\pi x)}{\cosh(\pi t)} \ll_{j, A} \left(\frac{x}{1+|t|}\right)^{\Re(2 i t)}(1 + |t|/x)^j, \quad -1/2 < \Re(2 i t) < A
\end{equation}
for $j, A \in \mathbb{N}_0 \defeq \mathbb{N} \cup \{0\}$.
Thus for $x \geq 1$ and $0 \leq j \leq 3$, we obtain
\[x^j H^{(j)}(x) \ll x^{\Re s+j} \int_{-\infty}^{\infty} \left(1 + \frac{|t|}{x}\right)^j |h(t) t|\, dt \ll x^{\Re s+j} \ll x^{-3/2},\]
and for $x < 1$, we shift the contour to $\Re(2 i t) = -\Re s + 10$ (not passing any pole by our assumption on $h$), getting
\[x^j H^{(j)}(x) \ll x^{\Re s} \int_{-\infty}^{\infty} \left(\frac{x}{1+|t|}\right)^{-\Re s + 10} \left(1 + \frac{|t|}{x}\right)^j \big|h\big(t - \textstyle\frac{i}{2}(10 - \Re s)\big) \, t\big|\, dt \ll  x.\qedhere\]
\end{proof}
  
\begin{lemma}\label{hard-K}  Let $s\in \mathbb{C}$ with $\Re s < 1$, and suppose that $h$ is even and holomorphic    in $|\Im t | < (-\Re s + 15)/2$, satisfying $h(t) \ll e^{-|t|}$ and having zeros at $\pm i(2n-1)/2$, $n \in \mathbb{N}$, in this region. \\
a) The transform $\mathscr{L}^+ \mathscr{K}_s h$, defined for $\Re s < -10$ by Lemma \ref{simple-K}, has analytic continuation to $\Re s< 1$, and we have the Sears--Titchmarsh inversion formula $\mathscr{L}^+ \mathscr{K}_0 h = h$.\\
b) We have the uniform bounds
\begin{displaymath}
\begin{split}
&\mathscr{L}^+ \mathscr{K}_s h(t)  \ll_{\Re s}  \int_{-\infty}^{\infty}    |h(\tau - \textstyle\frac{i}{2}(1-\Re s ))|(1+ |\tau|)  \displaystyle \prod_{\pm} (1 + |\textstyle\frac{1}{2}\Im s - \tau \pm \Re t|)^{-1 + \Re s} \,  d\tau
\end{split}
\end{displaymath}
 for $ t \in \mathbb{R} \cup [-i\theta, i\theta]$ and $\Re s<1$, and 
 \begin{displaymath}
\begin{split}
&\mathscr{L}^{\text{{\rm hol}}} \mathscr{K}_s h(k)  \ll_{\Re s}  \int_{-\infty}^{\infty}    |h(\tau - \textstyle\frac{i}{2}  \max( 2 - k - \Re s , 0))|(1+ |\tau|)  \displaystyle \prod_{\pm} (  |\textstyle\frac{1}{2}\Im s \pm \tau  | + k)^{-1 + \Re s} \,   d\tau
\end{split}
\end{displaymath}
for $k \in 2\mathbb{N}$, where the implied constants depend only on $\Re s$ (but not on $t$, $k$, $h$, $\Im s$). 
\end{lemma}  
  
\begin{proof}
For $\Re s < -10$, we have, by definition,
\[\mathscr{K}_s h(x) = x^s\frac{i}{\pi} \int_{\Re(i \tau) = \frac{1}{2}(10-\Re s)} \frac{ J_{2i\tau}(4\pi x)}{\cosh(\pi \tau)} h(\tau)  \tau  \, d\tau,\]
and we have an absolutely convergent double integral
\[\mathscr{L}^+ \mathscr{K}_s h(t) = -\int_0^{\infty}  \frac{1}{\sinh(\pi t)} (J_{2it}(4\pi x) - J_{-2it}(4\pi x)) x^s  \int_{\Re (i \tau) = \frac{1}{2}(10-\Re s )} \frac{ J_{2i\tau}(4\pi x)}{\cosh(\pi \tau)} h(\tau)  \tau   \, d\tau \,    \frac{dx}{x}.\]
To see the absolute convergence, we use \eqref{inparticular} with $j=0$ to bound $J_{\pm 2i t}(x)\ll_t  \min(x^{2\theta}, x^{-2\theta})$ for $t \in \mathbb{R} \cup [-i\theta, i\theta]$, and we combine  \eqref{inparticular} with the bound $J_{\nu}(x) \ll x^{-1/2}$ for $x \gg |\nu|^2$ (which follows from the asymptotic formula \cite[(8.451.1)]{GR}) to bound $J_{2i\tau}(x) \ll \min(x^{-\Re s + 10}, x^{-1/2} (1+|\tau|)^{-\Re s + 11}).$  We can compute the $x$-integral explicitly using  \cite[(6.574.2)]{GR}, getting
\begin{equation}\label{double}
\mathscr{L}^+ \mathscr{K}_s h(t) =  \int_{\Re (i \tau) = \frac{1}{2}(10-\Re s )}\frac{\Gamma(1-s) \cos(i\pi \tau + \textstyle \frac{1}{2}\pi s) h(\tau) \tau}{(2\pi)^{s}  \pi i  \cosh(\pi \tau)}\prod_{\pm}  \frac{\Gamma(\frac{s}{2} + i\tau \pm  it)}{\Gamma(1 - \frac{s}{2} +i\tau \pm it)}  \, d\tau. 
\end{equation}
Here we can put any $s$ with $\Re s < 1$ in the integrand (and also shift the contour to, say, $\Re (i \tau) = 5$),  in particular $s=0$, so that  
\[\mathscr{L}^+ \mathscr{K}_0 h(t) =  \int_{\Re (i \tau) = 5} \frac{ h(\tau) \tau}{ t^2-\tau^2} \,  \frac{d\tau}{\pi i}.\]
The integrand is odd, so the integral equals half the sum of the two residues at $\tau = \pm t$ and part (a) of the lemma follows. To prove part (b) for $\mathscr{L}^+ \mathscr{K}_s h(t) $, we shift the $\tau$-contour to $\Re(i\tau) = (1 - \Re(s))/2$  and estimate trivially in \eqref{double} using Stirling's formula. 
For $\mathscr{L}^{\textnormal{hol}} \mathscr{K}_s h(t) $,
  we have the similar expression
  \begin{displaymath}
  \begin{split}
  \mathscr{L}^{\textnormal{hol}} \mathscr{K}_s h(k) &=  i^{k+1}  \int_{\Re(i\tau) = a} \frac{\Gamma(1-s)(2\pi)^{-s}  \Gamma(\frac{k-1}{2} + \frac{s}{2} + i\tau)}{\Gamma(\frac{k+1}{2} - \frac{s}{2} - i\tau)\Gamma(\frac{3-k}{2} - \frac{s}{2} + i\tau)\Gamma(\frac{k+1}{2} - \frac{s}{2} + i\tau)}\frac{h(\tau)\tau}{\cosh(\pi \tau)}  \, d\tau\\
  \end{split}
  \end{displaymath}
  for any $a \in \mathbb{R}$ satisfying $k-1 + \Re s +2a> 0$, say $a = \max(\frac{2 - k - \Re s}{2}, 0).$ We can re-write this as 
\[\int_{\Re(i\tau) =a}\frac{\Gamma(1-s)  \cos(i\pi \tau + \textstyle \frac{1}{2}\pi s) h(\tau) \tau    }{(2\pi)^{s}  \pi i \cosh(\pi \tau)   } \prod_{\pm} \frac{\Gamma(\frac{k-1}{2} +\frac{s}{2}  \pm   i \tau))}{\Gamma(\frac{k+1}{2} - \frac{s}{2}  \pm  i \tau)) }  \, d\tau.\]
 The desired bound follows now from 
$\Gamma(z+w)/\Gamma(z) \ll_w (1+|z|)^{w}$ for $w\in \mathbb{R}$ and $|z|$ sufficiently large; see e.g.\ \cite[(8.328.2)]{GR}. 
\end{proof}

\section{Proof of Theorem \ref{thm1}}

\subsection{The set-up} We recall the definitions \eqref{defT} and \eqref{defTtriv} for a prime $q$ and integers $a, b$ satisfying $(ab, q) = (a, b) = 1$, $a\asymp b$,  
and an even holomorphic test function $F$ that is rapidly decaying on vertical lines and is divisible by $ (1-u)(v-1)^2 \prod_{j=1}^{50}(j-s)$. 
Initially we assume 
\begin{equation}\label{range}
2 < \Re s, \Re v < 3, \quad 10 < \Re u < 11.
\end{equation}
In this section, all implicit constants may depend on $s, u, v, \varepsilon$, and $F$. Additional dependencies will be mentioned. 
We proceed to define the correction polynomial $P_q(s, u, v, z)$, and to this end we define three auxiliary quantities
\begin{displaymath}
\begin{split}
 \mathcal{T}^{(1)}_{q}(s, u, v) & \defeq    \int_{\mathbb{R}} \zeta^{(q)}(s + z) \zeta^{(q)}(u+z) L^{(q)}(v - z, f) F(z) \, \frac{dz}{2\pi i},\\
 \mathcal{T}^{(2)}_{q}(s, u, v) & \defeq  q^{1-s-v}  \int_{\mathbb{R}}  \Bigl( \lambda(q) - \frac{1}{q^{v-z}}\Bigr)  \zeta(s + z) \zeta^{(q)}(u+z ) L(v - z, f) F(z) \, \frac{dt}{2\pi i},\\
 \mathcal{T}^{(3)}_{q}(s, u, v)  &\defeq   \int_{\mathbb{R}}  \big(q^{2-s-u-2v}   -  q^{1-u-2v + z} \big)  \zeta(s + z) \zeta(u+z) L(v - z, f) F(z) \, \frac{dz}{2\pi i}.  
\end{split}
\end{displaymath}
Here $\zeta^{(q)}(s) \defeq \zeta(s) \prod_{p \mid q}(1 - p^{-s})$ and $L^{(q)}(s,f) = L(s,f) \prod_{p \mid q} (1 - \lambda_f(p) p^{-s} + p^{-2s})$ are the usual $L$-functions with the Euler factors dividing $q$ omitted. We define
\[\mathcal{T}_q^{\textnormal{triv}}(s, u, v) \defeq \sum_{j=1}^3 \mathcal{T}^{(j)}_{q}(s, u, v),\]
so that \eqref{defTtriv} holds with 
\begin{equation}\label{defP}
\begin{split}
P_q(s, u, v, z) &= \left(1 - \frac{1}{q^{s+z}}\right)\left(1 - \frac{1}{q^{u+z}}\right)\left(1 - \frac{\lambda(q)}{q^{v-z}} + \frac{1}{q^{2v-2z}}\right) \\
&\qquad + q^{1-s-v} \Bigl( \lambda(q) - \frac{1}{q^{v-z}}\Bigr) \left(1 - \frac{1}{q^{u+z}}\right) + q^{2-s-u-2v} - q^{1-u-2v + z}.
\end{split}
\end{equation}
It is easy to see that this satisfies \eqref{P}. In the range \eqref{range}, we can open the Dirichlet series and obtain
\begin{displaymath}
\begin{split}
\mathcal{T}_{ a, b, q}(s, u, v) + \mathcal{T}_q^{(1)}(s, u, v)&= q\int_{(0)}  F(z) \sum_{\substack{(nmr, q) = 1\\ anm \equiv br \hspace{-.25cm} \pmod{q}}} \frac{\lambda(r)}{n^{s+z} m^{u+z} r^{v-z}} \, \frac{dz}{2\pi i},\\
\mathcal{T}_q^{(2)}(s, u, v) &= q^{1-s-v}   \int _{(0)} F(z)  \Bigl( \lambda(q) - \frac{1}{q^{v-z}}\Bigr)  \sum_{n, r, (m, q) = 1} \frac{\lambda(r)}{n^{s+z} m^{u+z} r^{v-z}} \,  \frac{dz}{2\pi i},\\
& = q \int_{(0)}  F(z)     \sum_{n, r, (m, q) = 1} \frac{\lambda(qr)}{(qn)^{s+z} m^{u+z} (qr)^{v-z}} \,  \frac{dz}{2\pi i},\\
 \mathcal{T}^{(3)}_q(s, u, v) &= q^{2-s-u-2v}  \int_{(0)}  F(z)(1 - q^{s+z-1})  \sum_{n, r, m} \frac{\lambda(r)}{n^{s+z} m^{u+z} r^{v-z}} \,  \frac{dz}{2\pi i}.
\end{split}
\end{displaymath}
We conclude that
\begin{displaymath}
\begin{split}
\tilde{\mathcal{T}}_{ a, b, q}(s, u, v) &\defeq \mathcal{T}_{ a, b, q}(s, u, v) + \mathcal{T}_q^{(1)}(s, u, v)+ \mathcal{T}_q^{(2)}(s, u, v)\\
& = q\int_{(0)}  F(z) \sum_{\substack{(m, q) = 1\\ n \equiv \overline{am}br \hspace{-.25cm} \pmod{q}}} \frac{\lambda(r)}{n^{s+z} m^{u+z} r^{v-z}} \, \frac{dz}{2\pi i}.\\
\end{split}
\end{displaymath}
Eventually the term $ \mathcal{T}^{(3)}_q(s, u, v)$ will remove the last coprimality condition $(m, q) = 1$, but this has to wait until the end of argument. Until then, we transform $\tilde{\mathcal{T}}_{ a, b, q}(s, u, v)$ and $ \mathcal{T}^{(3)}_q(s, u, v)$ in a parallel fashion. 

\subsection{Poisson summation} We write the $n$-sum in terms of the Hurwitz zeta function with $\alpha = \overline{am}br/q$ and shift the $z$-contour to the left to $\Re z = -4$. Our assumption on $F$ implies that the potential pole at $z = 1-s$ is cancelled. The $m, r$-sums are still absolutely convergent, and we apply the functional equation \eqref{hur}, getting
\begin{equation}\label{poisson1}
\begin{split}
\tilde{\mathcal{T}}_{ a, b, q}(s, u, v) &= \sum_{\pm} \int_{(-4)} F(z) G^{\pm}(1-s-z) q^{1-s-z} \sum_{n, r, (m, q) = 1} \frac{\lambda(r) e(\mp \overline{a} brn\overline{m}/q)}{n^{1-s-z} m^{u+z} r^{v-z}} \, \frac{dz}{2\pi i}
\end{split}
\end{equation}
and
\begin{equation}\label{poisson2}
\begin{split}
 \mathcal{T}^{(3)}_q(s, u, v) & = q^{2-s-u-2v}  \int_{(-4)}  F(z)(1 - q^{s+z-1}) G^{\pm}(1-s-z)  \sum_{n, r, m} \frac{\lambda(r)}{n^{1-s-z} m^{u+z} r^{v-z}} \,  \frac{dz}{2\pi i}\\
& = q^{2-s-u-2v}  \int_{(-4)}  F(z)  G^{\pm}(1-s-z)  \sum_{(n, q) = 1, r, m} \frac{\lambda(r)}{n^{1-s-z} m^{u+z} r^{v-z}} \,  \frac{dz}{2\pi i}. 
\end{split}
\end{equation}

\subsection{Reciprocity} 
  For $\alpha \in \mathbb{R} \setminus \{0\}$, we recall the absolutely convergent Mellin integral
\begin{equation*}
e(\alpha) = \int_{\mathcal{C}} G^{\text{sgn}(\alpha)}(s) |\alpha|^{-w} \, \frac{dw}{2\pi i},
\end{equation*}
where $\mathcal{C}$ is the contour 
\[\mathcal{C} = \textstyle (-\frac{3}{5} -i\infty, -\frac{3}{5} - i] \cup [-\frac{3}{5} - i, \frac{1}{10}] \cup [\frac{1}{10}, -\frac{3}{5} + i] \cup [-\frac{3}{5} + i, -\frac{3}{5} + i\infty).\]
In \eqref{poisson1}, we insert
\[1 = e\left(\pm \frac{brn}{amq}\right) \int_{\mathcal{C}}   G^{\pm}(w) \left( \frac{brn}{amq}\right)^{-w} \,  \frac{dw}{2\pi i}\]
and apply the additive reciprocity formula \eqref{addrec}. This gives the absolutely convergent expression
\begin{displaymath}
\begin{split}
\tilde{\mathcal{T}}_{ a, b, q}(s, u, v) &= \sum_{\pm} \int_{(-4)}  \int_{\mathcal{C}} F(z) \left(\frac{b}{a}\right)^{-w} G^{\pm}(w) G^{\pm}(1-s-z)  q^{1-s-z+w}\\
&\quad\quad\quad \times \sum_{n, (m, q) = 1, r}   \frac{\lambda(r)   e(\pm \overline{q}brn/(am))}{ n^{1-s-z+w}m^{u+z-w} r^{v-z+w}} \, \frac{dw}{2\pi i} \, \frac{dz}{2\pi i}. 
\end{split}
\end{displaymath}
We temporarily straighten the $\mathcal{C}$-contour to $\Re w = -3/5$, picking up the polar term
\[\sum_{\pm} \int_{(-4)}   F(z)   G^{\pm}(1-s-z)  q^{1-s-z}
  \sum_{n, (m, q) = 1, r}   \frac{\lambda(r)   e(\pm \overline{q}brn/(am))}{ n^{1-s-z}m^{u+z} r^{v-z}} \,   \frac{dz}{2\pi i}.\]
In the remaining double integral, we change variables $w \mapsto w+z$ (so that $\Re z = -4$, $\Re w = 17/5$), exchange the two integrals, and in the inner $z$-integral we bend the contour to the right to
\[\mathcal{C}(w) = \mathcal{C} - w =  \textstyle (-\frac{3}{5}-w -i\infty, -\frac{3}{5}-w - i] \cup [-\frac{3}{5} -w- i, \frac{1}{10} -w] \cup [\frac{1}{10}-w, -\frac{3}{5}-w + i] \cup [-\frac{3}{5} -w+ i, -\frac{3}{5} -w+ i\infty).\]
This picks up a polar term
\begin{displaymath}
\begin{split}
  -\sum_{\pm}  \int_{(17/5)}   F(-w)    G^{\pm}(1-s+w)  q^{1-s+w}  \sum_{n, (m, q) = 1, r}   \frac{\lambda(r)   e(\pm \overline{q}brn/(am))}{ n^{1-s+w}m^{u-w} r^{v+w}}  
 \, \frac{dw}{2\pi i}, \end{split}
\end{displaymath}
which cancels the previous one. This shows
\begin{equation}\label{tildeT}
\begin{split}
\tilde{\mathcal{T}}_{ a, b, q}(s, u, v)&=  \sum_{\pm} \int _{(17/5)} \Phi^{\pm}_{a, b, s}(w)  q^{1-s+w}\sum_{n, (m, q) = 1, r}   \frac{\lambda(r)   e(\pm \overline{q}brn/(am))}{ n^{1-s+w}m^{u-w} r^{v+w}}  \left(\frac{b}{a}\right)^{-w}    \, \frac{dw}{2\pi i},
\end{split}
\end{equation}
where
\begin{displaymath}
\begin{split}
\Phi^{\pm}_{a, b, s}(w) &=  \int_{\mathcal{C}(w)}  F(z) \left(\frac{b}{a}\right)^{-z} G^{\pm}(w+z) G^{\pm}(1-s-z)   \, \frac{dz}{2\pi i}\\
& = (2\pi)^{s-w-1} e^{\pm i\pi(w+1-s)/2}\int_{\mathcal{C}(w)}  F(z) \left(\frac{b}{a}\right)^{-z} \Gamma(w+z) \Gamma(1-s-z)  \, \frac{dz}{2\pi i}.
\end{split}
\end{displaymath}
Here we can straighten the contour and shift it to the far left to $\Re z = -A$, say. This gives a sum of polar terms of the shape
\[p_n(w) \defeq \frac{(-1)^n}{n!} (2\pi)^{s-w-1} e^{\pm i\pi(w+1-s)/2}F(-w-n) \left(\frac{b}{a}\right)^{w+n}  \Gamma(1-s+w+n), \quad n \in \mathbb{N}_0,\]
and a remaining integral that is holomorphic in $\Re w > -A$ and bounded by $\ll_{\Re w, A} (b/a)^A (1+|w|)^{\Re w - A - 1/2}$. Since $F(1-s) = \ldots =F(50-s) = 0$, we conclude that $\Phi^{\pm}_{a, b, s}$ is 
\begin{equation}\label{48}
 \text{holomorphic in   $|\Re w| < 48$ and satisfies $ \Phi^{\pm}_{a, b, s}(w) \ll (1 + |w|)^{-100} $}
   \end{equation}
   as long as $a\asymp b$ and $0 < \Re s < 3$. 
  We also observe that
\begin{equation}\label{zero}
  \sum_{\pm} \Phi^{\pm}_{a, b, s}(s) = 0.
\end{equation}

 Inserting 
\[1 = e\left(\pm \frac{brn}{am}\right) \int_{\mathcal{C}}   G^{\pm}(w) \left( \frac{brn}{am}\right)^{-w}  \, \frac{dw}{2\pi i}\]
into \eqref{poisson2}, we obtain in the same way
\[\mathcal{T}^{(3)}_q(s, u, v) =  \sum_{\pm} \int_{(17/5)}  \Phi^{\pm}_{a, b, s}(w)  q^{2-s-u-2v}\sum_{m, (n, q) = 1, r}   \frac{\lambda(r)   e(\pm  brn/(am))}{ n^{1-s+w}m^{u-w} r^{v+w}}  \left(\frac{b}{a}\right)^{-w}  \,  \frac{dw}{2\pi i}.\]
(The expression is still independent of $a, b$, even though the right-hand side seems to depend on $a, b$.)



\subsection{Poisson summation again}
 
We return to \eqref{tildeT}, split the $n$-sum into residue classes modulo $am$, express the $n$-sum in terms of the Hurwitz zeta function, shift the $w$-contour to $\Re w = 0$, 
and apply the functional equation \eqref{hur}, getting
\begin{equation}\label{tilde2}
\begin{split}
\tilde{\mathcal{T}}_{ a, b, q}&(s, u, v) = \sum_{\pm} \sum_{\sigma \in\{ \pm\}} \int _{(0)} \Phi^{\pm}_{a, b, s}(w) G^{-\sigma}(s-w) q^{1-s+w}\\
&\times\sum_{ (m, q) = 1, r, n} \sum_{\nu \hspace{-.25cm} \pmod{am}}e\left(\frac{\pm \overline{q}br\nu}{am}\right)    \frac{\lambda(r)  e(\sigma n\nu/(am)) }{ n^{s-w}m^{u-w} r^{v+w}}  \left(\frac{b}{a}\right)^{-w} (am)^{s-w-1}  \, \frac{dw}{2\pi i}. 
\end{split}
\end{equation}
Note that the possible pole at $w = s$ is cancelled by \eqref{zero}. 
Similarly,
\begin{equation}\label{tilde3}
\begin{split}
\mathcal{T}^{(3)}_q&(s, u, v)  = \sum_{\pm}  \sum_{\sigma \in\{ \pm\}}  \int _{(0)} \Phi^{\pm}_{a, b, s}(w) G^{-\sigma}(s-w) q^{2-s-u-2v} \\
&\times \sum_{  r, m, n} \sum_{\substack{\nu \hspace{-.25cm} \pmod{qam}\\ (\nu, q) = 1}}e\left(\frac{\pm br\nu}{am}\right)
    \frac{\lambda(r)  e(\sigma n\nu/(amq)) }{ n^{s-w}m^{u-w} r^{v+w}}  \left(\frac{b}{a}\right)^{-w} (amq)^{s-w-1}   \, \frac{dw}{2\pi i} .  
\end{split}
\end{equation}
 
\subsection{Vorono\u{\i} summation}

Our next aim is apply the functional equation for the $r$-sum. This requires some preparation because $b\nu$ is not necessarily coprime to $am$. Therefore we introduce various new variables.  We write $(m, b) = \beta_1$ and $b = \beta_1\beta_2$, $m = \beta_1m'$, $(m', \beta_2) = 1$.  Next, we write $(\nu, am') = m_1$, $(m_1, a) = \alpha_1$, and $a = \alpha_1\alpha_2$, $m_1 = \alpha_1m_1'$, $(m'_1, \alpha_2) = 1$ and further $m' = m_1' m_2$, $\nu = \alpha_1m_1' \nu'$, $(\nu', \alpha_2m_2) = 1$. Dropping the primes for notational simplicity, we recast the second line in \eqref{tilde2} as
\begin{displaymath}
\begin{split}
&  \sum_{\alpha_1\alpha_2 = a} \sum_{\beta_1\beta_2 =  b} \sum_{r, n} \sum_{\substack{(m_1, \alpha_2\beta_2q) = 1 \\ (m_2, q\beta_2) = 1}}    \sum_{\substack{\nu \hspace{-.25cm} \pmod{\alpha_2\beta_1m_2}\\ (\nu, \alpha_2m_2) = 1}} e\left(\frac{\pm \overline{q}\beta_2r\nu}{\alpha_2m_2}\right)   \frac{ \lambda(r)  e(\sigma n\nu/(\alpha_2\beta_1m_2)) }{ n^{s-w}(\beta_1m_1m_2)^{u-s+1} r^{v+w}} \frac{a^{s-1}}{b^w}
\end{split}
\end{displaymath}
and the second line in \eqref{tilde3} as
 \begin{displaymath}
\begin{split}
   &\sum_{\alpha_1\alpha_2 = a}\sum_{\beta_1\beta_2 =    b}  \sum_{r, n} \sum_{\substack{(m_1, \alpha_2\beta_2 q) = 1\\ (m_2, \beta_2) = 1}} \sum_{\substack{\nu \, (q\alpha_2\beta_1m_2)\\ (\nu, \alpha_2m_2q) = 1}}e\left(\frac{\pm \beta_2r\nu}{\alpha_2m_2}\right)   \frac{\lambda(r)  e(\sigma n\nu/(\alpha_2\beta_1m_2q)) }{ n^{s-w}(\beta_1m_1m_2)^{u-s+1} r^{v+w}}  \frac{a^{s-1}}{b^w}.
 \end{split}
\end{displaymath}
Note that in both cases the $m_1$-sum is $\zeta^{(\alpha_2\beta_2q)}(1 + u - s)$.  Both terms are   now in shape to apply Vorono\u{\i} summation. We express the $r$-sum in terms of the twisted $L$-function \eqref{twistedL}, shift the $w$-contour to $\Re w = -4$ (picking up a possible residue at $w = 1-v$), and apply the functional equation \eqref{vor}. This gives
 \begin{equation}\label{main1}
\begin{split}
\tilde{\mathcal{T}}_{ a, b, q}&(s, u, v) = \sum_{\pm} \sum_{\sigma, \tau  \in\{ \pm\}} \int _{(-4)} \Phi^{\pm}_{a, b, s}(w) G^{-\sigma}(s-w) q^{1-s+w}\sum_{\alpha_1\alpha_2 = a} \sum_{\beta_1\beta_2 =  b} \zeta^{(\alpha_2\beta_2q)}(1 + u - s) \\
&\times G^{-\tau}_f(1-v-w) \sum_{r, n}  \sum_{  (m_2, q\beta_2) = 1}    \sum_{\substack{\nu \hspace{-.25cm} \pmod{\alpha_2\beta_1m_2}\\ (\nu, \alpha_2m_2) = 1}} e\left(\frac{\pm \tau q\overline{\beta_2\nu}r}{\alpha_2m_2}\right)   \frac{ \lambda(r)  e(\sigma n\nu/(\alpha_2\beta_1m_2)) }{ n^{s-w}(\beta_1 m_2)^{u-s+1} r^{1-v-w}}\\
&\times  \frac{a^{s-1}}{(\beta_1\beta_2)^w} (\alpha_2m_2)^{1-2v-2w} \, \frac{dw}{2\pi i}  + \mathcal{P}^{(1)}_{a, b, q}(s, u, v)\\
\end{split}
\end{equation}
where the polar term $\mathcal{P}^{(1)}_{a, b, q}(s, u, v)$ vanishes unless $f$ is Eisenstein, in which case it equals 
\begin{equation}\label{PP1}
\begin{split}
\mathcal{P}^{(1)}_{a, b, q}(s, u, v) = \Res_{w=1-v} &  \sum_{\pm} \sum_{\sigma  \in\{ \pm\}} \Phi^{\pm}_{a, b, s}(w)G^{-\sigma}(s-w) q^{1-s+w}\sum_{\alpha_1\alpha_2 = a} \sum_{\beta_1 \beta_2 = b}  \zeta^{(\alpha_2\beta_2q)}(1 + u - s) \\
&\times \sum_{\beta_1 \mid  n} \sum_{(m_2 , q\beta_2) = 1} \frac{a^{s-1}\beta_1 r_{\alpha_2m_2}(n/\beta_1)}{(\beta_1m_2)^{u-s+1}n^{s-w}b^w} \textstyle L(v+w, \frac{\ast}{\alpha_2m_2}, f) , 
\end{split}
\end{equation}
where $r_c(n)$ denotes the Ramanujan sum. (Recall that by \eqref{laurent} the residue is independent of the numerator $\ast$ in the twist of the $L$-function.)  Similarly,
\begin{equation}\label{main2}
\begin{split}
\mathcal{T}^{(3)}_q&(s, u, v)  =  \sum_{\pm} \sum_{\sigma, \tau  \in\{ \pm\}}   \int _{(-4)}  \Phi^{\pm}_{a, b, s}(w) G^{-\sigma}(s-w)  q^{1-w -u-2v} \sum_{\alpha_1\alpha_2 = a}\sum_{\beta_1\beta_2 =    b}  \zeta^{(\alpha_2\beta_2q)}(1 + u - s) \\
& \times G^{-\tau}_f(1-v-w)  \sum_{r, n}  \sum_{  (m_2, \beta_2) = 1} \sum_{\substack{\nu \hspace{-.25cm} \pmod{q\alpha_2\beta_1m_2}\\ (\nu, \alpha_2m_2q) = 1}}e\left(\frac{\pm \tau \overline{\beta_2\nu}r}{\alpha_2m_2}\right)   \frac{\lambda(r)  e(\sigma n\nu/(\alpha_2\beta_1m_2q)) }{ n^{s-w}(\beta_1 m_2)^{u-s+1} r^{1-v-w}}\\
& \times  \frac{a^{s-1}}{b^w}  (\alpha_2m_2)^{1-2v-2w} \, \frac{dw}{2\pi i}  + \mathcal{P}^{(2)}_{a, b, q}(s, u, v),
\end{split}
\end{equation}
where 
\begin{displaymath}
\begin{split}
\mathcal{P}^{(2)}_{a, b, q}(s, u, v) = \Res_{w=1-v} &  \sum_{\pm} \sum_{\sigma  \in\{ \pm\}} \Phi^{\pm}_{a, b, s}(w)G^{-\sigma}(s-w) q^{1-w-u-2v}\sum_{\alpha_1\alpha_2 = a} \sum_{\beta_1 \beta_2 = b}  \zeta^{(\alpha_2\beta_2q)}(1 + u - s) \\
&\times \sum_{\beta_1\mid  n} \sum_{(m_2 , \beta_2) = 1} \frac{a^{s-1}\beta_1 r_{\alpha_2m_2q}(n/\beta_1)}{(\beta_1m_2)^{u-s+1}n^{s-w}b^w} \textstyle L(v+w, \frac{\ast}{\alpha_2m_2}, f) .
\end{split}
\end{displaymath}
We will compute the two polar terms in a moment, but we observe already at this point that now the time has come to combine the two main terms. Indeed, the main term in \eqref{main2} simply counteracts the condition $(m_2, q)=1$ of the main term in \eqref{main1} and supplies the missing terms $q \mid m$. Combining the two, we see that
\[\tilde{\mathcal{T}}_{ a, b, q}(s, u, v) + 
\mathcal{T}^{(3)}_q(s, u, v) =  \mathcal{T}^{\ast}_{ a, b, q}(s, u, v)+ \sum_{j=1}^2\mathcal{P}^{(j)}_{a, b, q}(s, u, v),\]
where
\begin{equation}\label{tast}
\begin{split}
 \mathcal{T}^{\ast}_{ a, b, q}(s, u, v) = & \sum_{\pm} \sum_{\sigma, \tau  \in\{ \pm\}}   \int _{(-4)} \Phi^{\pm}_{a, b, s}(w) G^{-\sigma}(s-w)  G^{-\tau}_f(1-v-w)  q^{1-w -u-2v} \\
 &\times \sum_{\alpha_1\alpha_2 = a}\sum_{\beta_1\beta_2 =    b}  \zeta^{(\alpha_2\beta_2q)}(1 + u - s)   \frac{\alpha_1^{s-1}\alpha_2^{s-2v-2w}}{\beta_1^{u-s+1+w}\beta_2^w}  \\
&\times \sum_{r, n}\sum_{  (m_2, \beta_2) = 1}    \sum_{\substack{\nu \hspace{-.25cm} \pmod{\alpha_2\beta_1m_2}\\ (\nu, \alpha_2m_2) = 1}} e\left(\frac{\pm \tau q\overline{\beta_2\nu}r}{\alpha_2m_2}\right)   \frac{ \lambda(r)  e(\sigma n\nu/(\alpha_2\beta_1m_2)) }{ n^{s-w}  m_2^{u-s+2v+2w} r^{1-v-w}}  \, \frac{dw}{2\pi i} .
\end{split}
\end{equation}

\subsection{Computation of polar terms}

In this subsection, we compute $\mathcal{P}^{(j)}_{a, b, q}(s, u, v) $ for $j = 1, 2$. 
We consider first 
\begin{displaymath}
\begin{split}
\sum_{\alpha_1\alpha_2 = a} \sum_{\beta_1 \beta_2 = b}  \zeta^{(\alpha_2\beta_2q)}(1 + u - s)  \sum_{  n} \sum_{(m_2 , q\beta_2) = 1} \frac{a^{s-1}\beta_1 r_{\alpha_2m_2}(n)}{(\beta_1m_2)^{u-s+1}(\beta_1 n)^{s-w}b^w} \textstyle L(v+w, \frac{\ast}{\alpha_2m_2}, f)
\end{split}
\end{displaymath}
corresponding to the last four sums in \eqref{PP1} for $w$ in a neighbourhood of $1-v$. 
Substituting 
\[r_{\alpha_2m_2}(n) = \sum_{\substack{d_1 d_2 = \alpha_2 m\\ d_1 \mid n}} d_1 \mu(d_2),\]
we obtain
\begin{displaymath}
\begin{split}
& \zeta(s-w)\sum_{\alpha_1\alpha_2 = a} \sum_{\beta_1 \beta_2 = b} \zeta^{(\alpha_2\beta_2q)}(1 + u - s)   \frac{a^{s-1}}{\beta_1^{u} \beta_2^w}   \sum_{(m_2 , q\beta_2) = 1}\sum_{d_1d_2 = \alpha_2m_2} \frac{\mu(d_2) }{d_1^{s-w-1}} \frac{ 1}{   m_2^{u-s+1}}  \textstyle L(v+w, \frac{\ast}{\alpha_2m_2}, f)\\
 &=
 \zeta(s-w)\sum_{\alpha_1\alpha_2 = a} \sum_{\beta_1 \beta_2 = b} \zeta^{(\alpha_2\beta_2q)}(1 + u - s) \frac{a^{s-1}}{\beta_1^{u} \beta_2^w}   \sum_{\substack{\alpha_2 \mid d_1d_2 \\ (d_1d_2/\alpha_2, q\beta_2) = 1}} \frac{\mu(d_2) }{d_1^{s-w-1}} \frac{ \alpha_2^{u-s+1}}{   (d_1d_2)^{u-s+1}} \textstyle L(v+w, \frac{\ast}{d_1d_2}, f).
\end{split}
\end{displaymath}
We write $(d_1, \alpha_2) = A_1$, $A_1A_2 = \alpha_2$, $A_2 \mid d_2$, getting eventually
\begin{equation}\label{P1}
\begin{split}
\mathcal{P}^{(1)}_{a, b, q}(s, u, v) &= \Res_{w=1-v}   \sum_{\pm} \sum_{\sigma  \in\{ \pm\}} \Phi^{\pm}_{a, b, s}(w)G^{-\sigma}(s-w) q^{1-s+w} \zeta(s-w)\\
&\times \sum_{\alpha_1A_1A_2 = a} \sum_{ \beta_1 \beta_2 = b } \frac{ \zeta^{(A_1A_2\beta_2q)}(1 + u - s) \mu(A_2)A_2^w  }{\beta_1^{u} \beta_2^w (\alpha_1A_2)^{1-s} }    \sum_{(d_1d_2, qA_2\beta_2) = 1}   \frac{ \mu(d_2)L(v+w, \frac{\ast}{A_1A_2d_1d_2}, f)}{    d_1^{u-w} d_2^{u-s+1}}  .
\end{split}
\end{equation}
By \eqref{laurent}, this is a  linear combination of
\[\sum_{\alpha_1A_1A_2 = a} \sum_{ \beta_1 \beta_2 = b } \Phi_{a, b, q}^{\pm}(1-v) G^{-\sigma}(s+v-1) q^{2-v-s} \zeta(s+v-1)\frac{\zeta^{(A_1A_2\beta_2 q)}(1-s+u)\zeta^{(qA_2\beta_2)}(u+v)}{\zeta^{(qA_2\beta_2)}(2 + u - s) \alpha_1^{1-s}A_2^{1+v-s} A_1 \beta_1^u\beta_2^{1-v}}\]
and derivatives thereof.  The same computation shows 
\begin{equation}\label{P2}
\begin{split}
\mathcal{P}^{(2)}_{a, b, q}(s, u, v) &= \Res_{w=1-v}   \sum_{\pm} \sum_{\sigma  \in\{ \pm\}} \Phi^{\pm}_{a, b, s}(w)G^{-\sigma}(s-w) q^{1-w-u-2v} \zeta(s-w)\\
&\times  \sum_{q_1q_2 = q} \frac{ 1}{q_1^{s-w-1}}\sum_{\alpha_1A_1A_2 = a } \sum_{\beta_1 \beta_2 = b} \frac{\zeta^{(A_1A_2\beta_2 q)}(1-s+u)\mu(A_2q_2) A_1^w }{\beta_1^{u} \beta_2^w(\alpha_1A_2)^{1-s} } \\
& \times  \sum_{\substack{(d_1, A_2\beta_2q_2) =1\\(d_2, A_2\beta_2) = 1} }   \frac{\mu(d_2)L(v+w, \frac{\ast}{A_1A_2d_1d_2}, f) }{ d_1^{u-w}d_2^{u-s+1} }, \end{split}
\end{equation}
which is a linear combination of 
\begin{displaymath}
\begin{split}
  \sum_{q_1q_2 = q} &\sum_{\alpha_1A_1A_2 = a} \sum_{ \beta_1 \beta_2 = b } \Phi_{a, b, q}^{\pm}(1-v) G^{-\sigma}(s+v-1)\zeta(s+v-1)\\
  &\times \frac{  q_1^{2-s-u-2v}}{q_2^{u+v}}  \frac{\zeta^{(A_1A_2\beta_2 q)}(1-s+u)\zeta^{(A_2\beta_2q_2)}(u+v-1)}{\zeta^{(A_2\beta_2)}(2 + u - s) \alpha_1^{1-s}A_2^{2-s} A_1^{v} \beta_1^u\beta_2^{1-v}}
  \end{split}
  \end{displaymath}
and derivatives thereof. 

\subsection{Application of the Kuznetsov formula}
We return to \eqref{tast} and recognize the $\nu$-sum as a Kloosterman sum. More precisely, the $\nu$-sum vanishes unless $\beta_1 \mid n$, so that the second and third line of \eqref{tast} equal
\begin{equation}\label{23}
\sum_{\alpha_1\alpha_2 = a} \sum_{\beta_1\beta_2 =  b}  \zeta^{(\alpha_2\beta_2q)}(1 + u - s) \frac{\alpha_1^{s-1}\alpha_2^{s-2v-2w}}{\beta_1^{u}\beta_2^w} \sum_{r, n}\sum_{  (m_2, \beta_2) = 1}  \frac{ \lambda(r)  S(\pm \tau q \overline{\beta_2} r,  \sigma n, \alpha_2m_2)      }{ n^{s-w}  m_2^{u-s+2v+2w} r^{1-v-w}}. 
\end{equation}
For $(\beta_2, q\alpha_2m) = 1$, we have, by the twisted multiplicativity of Kloosterman sums,
\[S(\pm \tau qr, \sigma \beta_2 n, \beta_2\alpha_2 m) = r_{\beta_2}(r)  S(\pm \tau qr \overline{\beta_2}, \sigma n, \alpha_2m).\]
At this point,   we use the fact that $b$ is squarefree; in particular, the Ramanujan sum $r_{\beta_2}(r)$ does not vanish. Write $B_1 = (\beta_2, r)$, $B_2B_1 = \beta_2$, $r = r' B_1$, $(r', B_2) = 1$.   Then
\[S(\pm \tau qr \overline{\beta_2}, \sigma n, \alpha_2m) = \frac{1}{r_{\beta_2}(r)} S(\pm \tau qr'B_1, \sigma B_1B_2 n, B_1B_2\alpha_2 m) = \frac{\phi(B_1)}{\phi(B_1) \mu(B_2)}S(\pm \tau qr' ,  \sigma B_2 n,  B_2\alpha_2 m),\]
so that \eqref{23} is equal to
  \begin{displaymath}
\begin{split}
&\sum_{\alpha_1\alpha_2 = a} \sum_{\beta_1B_1B_2=  b}  \frac{\zeta^{(\alpha_2B_1B_2q)}(1 + u - s) \mu(B_2)\alpha_1^{s-1}\alpha_2^{s-2v-2w}}{\beta_1^u B_2^w B_1^{1-v} } \sum_{(r,B_2) = 1,  n}\sum_{  (m_2, \beta_2) = 1}  \frac{ \lambda(rB_1)  S(\pm \tau q   r, \sigma B_2 n, B_2\alpha_2m_2)      }{ n^{s-w}  m_2^{u-s+2v+2w} r^{1-v-w}}.  
\end{split}
\end{displaymath}
Here we can drop the condition $(m_2, B_2) = 1$, since otherwise the Kloosterman sum vanishes (since $(r, B_2) = 1)$. We remove the remaining condition $(m_2, B_1) = 1$ by M\"obius inversion, getting
 \begin{displaymath}
\begin{split}
\sum_{\alpha_1\alpha_2 = a} &\sum_{\beta_1B_3B_4B_2=  b} \frac{\zeta^{(\alpha_2B_2B_3B_4q)}(1 + u - s)  \mu(B_2)\mu(B_3)\alpha_1^{s-1}\alpha_2^{s-2v-2w}}{\beta_1^u B_2^w  B_3^{1+u-s+v+2w} B_4^{1-v} } \\
&\times \sum_{(r,B_2) = 1,  n}\sum_{ m_2}  \frac{ \lambda(rB_3B_4)  S(\pm\tau q   r, \sigma B_2 n, B_2B_3\alpha_2m_2)      }{ n^{s-w}  m_2^{u-s+2v+2w} r^{1-v-w}}. 
\end{split}
\end{displaymath}
Re-arranging, we   obtain the final expression
  \begin{equation}\label{final1}
\begin{split}
 \mathcal{T}^{\ast}_{ a, b, q}(s, u, v) = & \sum_{\pm} \sum_{\sigma, \tau  \in\{ \pm\}}  q^{(3-s-u-2v)/2}  \sum_{\alpha_1\alpha_2 = a} \sum_{\beta_1B_3B_4B_2=  b} \frac{\mu(B_2)\mu(B_3)\alpha_1^{s-1}\alpha_2^{u}\zeta^{(\alpha_2B_2B_3B_4q)}(1-s+u) }{\beta_1^u B_2^{(s-u-1-2v)/2}  B_3^{1-v } B_4^{1-v} }    \\
 & \times \sum_{(r,B_2) = 1,  n} \frac{\lambda(rB_3B_4)}{n^{(s+u-1+2v)/2} r^{(1-s+u)/2}}    \sum_{  B_2B_3\alpha_2 \mid m_2}  \frac{    S(\pm\sigma \tau  q   r,  B_2 n,  m_2)      }{  m_2  } \Psi^{\pm, \sigma, \tau}_{a, b, s, u, v}\left(\frac{\sqrt{qrB_2n}}{m_2}\right),  
\end{split}
\end{equation}
where
\begin{displaymath}
\begin{split}
\Psi(x) &= \Psi^{\pm, \sigma, \tau}_{a, b, s, u, v}(x)  = x^{u-s+2v -1}\int_{(-4)}  \Phi^{\pm}_{a, b, s}(w) G^{-\tau}_f(1-v-w) G^{-\sigma}(s-w) x^{2w} \, \frac{dw}{2\pi i}\\
&= \int_{(0)} \Phi^{\pm}_{a, b, s}\Big(\frac{1+s-u-2v-w}{2}\Big) G^{-\tau}_f\Big(\frac{1-s+u+w}{2}\Big) G^{-\sigma}\Big(\frac{s+u+2v-1+w}{2}\Big) x^{-w} \, \frac{dw}{2\pi i}.
\end{split}
\end{displaymath}
In the region \eqref{range}, the integrand is holomorphic in  $2\theta - 8 < \Re w  < 34$ (recalling \eqref{Gf} and \eqref{48}) and rapidly decaying on vertical lines; in particular, the assumption $x^j\Psi^{(j)}(x) \ll \min(x, x^{-3/2})$ for $0 \leq j \leq 3$ of the Kuznetsov formula \eqref{kuz2} is satisfied.     
By \eqref{kuz2}, the $m_2$-sum equals 
\[\mathcal{A}_B^{\textnormal{Maa{\ss}}}(\epsilon qr, B_2r; \mathscr{L}^{\epsilon}\Psi) +  \mathcal{A}_B^{\textnormal{Eis}}(\epsilon qr, B_2r; \mathscr{L}^{\epsilon}\Psi)+   \mathcal{A}_B^{\textnormal{hol}}(\epsilon qr, B_2r; \mathscr{L}^{\textnormal{hol}}\Psi)\]
with $B = B_2B_3\alpha_2$ and $\epsilon = \pm \sigma \tau$. In the larger region
 \begin{equation}\label{larger}
 1/2 \leq \Re s, \Re v  \leq 3, \quad \Re s \leq \Re u \leq 11,
 \end{equation}
 the integrand of $\Psi$ is holomorphic in $2\theta - 1 < \Re w  < 32$ (and meromorphic in $|\Re w| < 32$) and rapidly decaying on vertical lines.  By \cite[Lemma 3a]{BK1} and \eqref{48}, we conclude that uniformly in this region,
 \begin{equation}\label{trafobound}
 \mathscr{L}^{\pm}\Psi(t) \ll (1+|t|)^{-30}, \quad \mathscr{L}^{\textnormal{hol}}\Psi(k) \ll  k^{-30}
 \end{equation}
as long as $a \asymp b$.

\subsection{The cuspidal contribution}  We start with the analysis of the Maa{\ss} spectrum. Inserting the definitions \eqref{aa1} and \eqref{rho-cusp} and using the notations and conventions of Section \ref{kuznetsov}, we obtain
\begin{displaymath}
\begin{split}
&\mathcal{A}_B^{\textnormal{Maa{\ss}}}(\epsilon qr, B_2r; \mathscr{L}^{\epsilon}\Psi)  = \sum_{B_0\mid B} \sum_{\psi\in \mathcal{B}^{\ast}(B_0)} \sum_{M \mid \frac{B}{B_0}} \rho_{\psi, M, B}(\epsilon qr) \rho_{\psi, M, B}(B_2n)  \mathscr{L}^{\epsilon}\Psi(t_\psi) \\
&= \sum_{B_0\mid B} \sum_{\psi\in \mathcal{B}^{\ast}(B_0)} \epsilon^{(1-\epsilon)/2}_{\psi}\sum_{M \mid \frac{B}{B_0}}  \frac{\prod_{p \mid B_0}(1 - p^{-2})}{L(1, \mathrm{Ad}^2 \psi) B\nu(B)}  \sum_{d_1, d_2 \mid M} \xi_{\psi}(M, d_1)  \xi_{\psi}(M, d_2) \frac{d_1d_2}{M} \lambda_{\psi}\Big(\frac{qr}{d_1}\Big) \lambda_\psi\Big(\frac{B_2n}{d_2}\Big)  \mathscr{L}^{\epsilon}\Psi(t_\psi).
\end{split}
\end{displaymath}
Summing over $n$ and $r$ as in \eqref{final1}, we obtain
 \begin{displaymath}
\begin{split}
&\sum_{B_0\mid B}  \sum_{\psi\in \mathcal{B}^{\ast}(B_0)} \epsilon^{(1-\epsilon)/2}_{\psi} \sum_{M \mid \frac{B}{B_0}}  \frac{\prod_{p \mid B_0}(1 - p^{-2})}{L(1, \mathrm{Ad}^2 \psi) B \nu(B)}  \sum_{d_1, d_2 \mid M} \xi_{\psi}(M, d_1)  \xi_{\psi}(M, d_2) \frac{d_1d_2}{M} \\
&\times \sum_{(r, B_2) = 1}  \frac{\lambda(rB_3B_4)\lambda_\psi(qr/d_1) }{  r^{(1-s+u)/2}}  \sum_n\frac{  \lambda_\psi(B_2n/d_2)}{n^{(s+u-1+2v)/2}  } \mathscr{L}^{\epsilon}\Psi(t_\psi). 
\end{split}
\end{displaymath}
Since $(q, B) = 1$, we have $(d_1, q) = 1$, and so by \eqref{mult}, the $r$-sum equals
\begin{displaymath}
\begin{split}
&\frac{\delta_{(d_1, B_2) = 1}}{d_1^{(1-s+u)/2}}\sum_{(r, B_2) = 1}  \frac{\lambda (rd_1B_3B_4)\lambda_\psi(qr) }{  r^{(1-s+u)/2}} \\
& =  \frac{\delta_{(d_1, B_2) = 1}}{d_1^{(1-s+u)/2}} \sum_{\substack{\delta_1 \mid d_1B_3B_4\\ (\delta_1, B_0) = 1}} \frac{\mu(\delta_1)\lambda(d_1B_3B_4/\delta_1)}{\delta_1^{(1-s+u)/2}}\sum_{\delta_2 \mid q} \frac{\mu(\delta_2)\lambda_\psi(q/\delta)}{\delta_2^{(1-s+u)/2}} \frac{L^{(B_2)}(f \times \psi, (1-s+u)/2)}{\zeta^{(B_2)}(1-s+u)}.
\end{split}
\end{displaymath}
Similarly, the $n$-sum equals
\begin{displaymath}
\begin{split}
&\left( \frac{(B_2, d_2)}{d_2}\right)^{(s+u-1+2v)/2}\sum_n \frac{\lambda_\psi(B_2n/(d_2, B_2))}{n^{(s+u-1+2v)/2}  } \\
&  = \left( \frac{(B_2, d_2)}{d_2}\right)^{(s+u-1+2v)/2}\sum_{\substack{B^{\ast} \mid B_2/(d_2, B_2)\\ (B^{\ast}, B_0) = 1}} \frac{\mu(B^{\ast}) \lambda_\psi(B_2/((d_2, B_2) B^{\ast}))}{(B^{\ast})^{(s+u-1+2v)/2}} L(\psi, (s+u-1+2v)/2).
\end{split}
\end{displaymath}
Putting everything together,  the Maa{\ss} contribution to \eqref{final1} equals
 \begin{displaymath}
\begin{split}
\sum_{A \mid ab} \sum_{\psi \in \mathcal{B}^{\ast}(A)} \Theta^{\textnormal{Maa{\ss}}}_{a, b, q}(s, u, v, \psi) \frac{L(\frac{s+u-1+2v}{2}, \psi) L( \frac{1-s+u}{2}, f \times \psi)}{L(1, \mathrm{Ad}^2 \psi)},
\end{split}
\end{displaymath}
  where
\begin{equation}\label{defTheta}
\begin{split}
&\Theta^{\textnormal{Maa{\ss}}}_{a, b, q}(s, u, v, \psi)  \defeq\sum_{\pm} \sum_{\sigma, \tau  \in\{ \pm\}} \epsilon^{(1\mp\sigma\tau)/2}_{\psi}  q^{(3-s-u-2v)/2}  \underset{A \mid B_2B_3\alpha_2}{\sum_{\alpha_1\alpha_2 = a} \sum_{\beta_1B_3B_4B_2=  b}} \frac{\mu(B_2)\mu(B_3)\alpha_1^{s-1}\alpha_2^{u}  }{\beta_1^u B_2^{(s-u-1-2v)/2}  B_3^{1-v } B_4^{1-v} }  \\
&\times \frac{\prod_{p \mid A}(1 - p^{-2})}{  B_2B_3\alpha_2 \nu(B_2B_3\alpha_2)} \sum_{M \mid \frac{B_2B_3\alpha_2}{A}}  \sum_{\substack{d_1, d_2 \mid M\\ (d_1, B_2) = 1}}  \frac{\xi_{\psi}(M, d_1)  \xi_{\psi}(M, d_2)d_1d_2  (  (B_2, d_2)/d_2 )^{\frac{s+u-1+2v}{2}} }{M d_1^{(1-s+u)/2}L_{B_2}(f \times \psi, (1-s+u)/2) \zeta_{\alpha_2B_3B_4q}(1-s+u) }    \\
 &\times \sum_{\substack{\delta_1 \mid d_1B_3B_4\\ (\delta_1, A) = 1}} \frac{\mu(\delta_1)\lambda(d_1B_3B_4/\delta_1)}{\delta_1^{(1-s+u)/2}}\sum_{\delta_2 \mid q} \frac{\mu(\delta_2)\lambda_\psi(q/\delta)}{\delta_2^{(1-s+u)/2}}\sum_{\substack{B^{\ast} \mid \frac{B_2}{(d_2, B_2)}\\ (B^{\ast}, A) = 1}} \frac{\mu(B^{\ast}) \lambda_\psi(B_2/((d_2, B_2) B^{\ast}))}{(B^{\ast})^{(s+u-1+2v)/2}}   \mathscr{L}^{\pm \sigma \tau}\Psi^{\pm, \sigma, \tau}_{a, b, s, u, v}(t_\psi)
 \end{split}
\end{equation}
for $\psi \in \mathcal{B}^{\ast}(A)$ of spectral parameter $t_{\psi}$ and parity $\epsilon_{\psi}$. 
Clearly this expression is holomorphic in the region \eqref{larger}. We proceed to confirm the bound \eqref{boundTheta} for $\Re s = \Re u = \Re v = 1/2$. This requires a little more than a trivial bound of \eqref{defTheta}. The critical variable is $B_2$. In order to get enough saving, we need to exploit some cancellation. To this end, we write $M = M_1 M_2$, where $(M_1, B_2) = 1$ and $M_2 \mid B_2$. (Recall that $ab$ is squarefree.) Since $(d_1, B_2) = 1$, we have $d_1 \mid M_1$, and we write $d_2 = d_2' d_2''$ with $d_2' \mid M_1$, $d_2'' \mid M_2$. In this way, the $M_2$-sum becomes
\[\sum_{d_2'' \mid M_2 \mid \frac{B_2}{(A, B_2)}}\frac{ \xi_{\psi}(M_2, 1)  \xi_{\psi}(M_2, d''_2) d''_2  (  (B_2, d''_2)/d''_2 )^{\frac{s+u-1+2v}{2}} }{M_2    } \sum_{\substack{B^{\ast} \mid \frac{B_2}{(d''_2, B_2)}\\ (B^{\ast}, A) = 1}} \frac{\mu(B^{\ast}) \lambda_\psi(B_2/((d''_2, B_2) B^{\ast}))}{(B^{\ast})^{(s+u-1+2v)/2}}.\]
If $B_2 \mid A$, this is equal to $\lambda_{\psi}(B_2) \ll B_2^{-1/2}$ by \eqref{stein}. If $B_2 \nmid A$, this is equal to
\[\prod_{p \mid B_2}\left(\left( \lambda_{\psi}(p) - \frac{1}{p^{(s+u-1+2v)/2}}\right) \left(1  + \frac{\xi_{\psi}(p, 1)^2}{p}\right)  + \xi_{\psi}(p, 1) \xi_{\psi}(p, p)\right).\]
By \eqref{xi-explicit}, the leading term $\lambda_{\psi}(p)$ cancels (to first order approximation), and each $p$-factor in the preceding display is bounded by  $p^{-1/2} + p^{3\theta-1} \ll p^{-1/2}$ for $\Re s = \Re u = \Re v = 1/2$. Hence in all cases the $M_2$-sum is $\ll B_2^{-1/2 + \varepsilon}$. Combining with \eqref{trafobound} and \eqref{xi-arithmetic}, we obtain
\begin{displaymath}
\begin{split}
\Theta^{\textnormal{Maa{\ss}}}_{a, b, q}(s, u, v, \psi) & \ll \frac{(abq )^{\varepsilon}q^{1/2}}{(1 + |t_{\psi}|)^{30}} \underset{A \mid B_2B_3\alpha_2}{\sum_{\alpha_1\alpha_2 = a} \sum_{\beta_1B_3B_4B_2=  b}} \sum_{\substack{d_1, d_2 \mid M \mid \frac{B_2B_3\alpha_2}{A}\\ (M, B_2) = 1}}   \frac{d_1^{1/2} d_2^{1/2 - \theta} (1+|\lambda_{\psi}(q)|)}{
\alpha_2^{1/2} B_3^{3/2-\theta} M^{1-2\theta}B_2^{1/2}}\\
\end{split}
\end{displaymath}
for $\Re s = \Re u = \Re v = 1/2$. This is increasing in $d_1, d_2$, and the result is increasing in $M$, so  that one easily confirms \eqref{boundTheta}. 

The same formula holds for the holomorphic contribution to \eqref{final1}, except that the transform $  \mathscr{L}^{\pm \sigma \tau}\Psi^{\pm, \sigma, \tau}_{a, b, s, u, v}(t_\psi)$ has to be replaced with $  \mathscr{L}^{\textnormal{hol}}\Psi^{\pm, \sigma, \tau}_{a, b, s, u, v}(k_\psi)$ and $\epsilon_{\psi}  = 0$ if $\pm \sigma \tau = -1$. The corresponding  bound \eqref{thetahol} is even simpler to obtain because $\theta = 0$ in the holomorphic case. 

\subsection{The   Eisenstein contribution}  
 By \eqref{rho-eis}, we have 
\begin{equation}\label{eis-raw}
\begin{split}
\mathcal{A}_{B_2B_3\alpha_2}^{\textnormal{Eis}}&(\epsilon qr, B_2r; \mathscr{L}^{\epsilon}\Psi)  = \int_{\mathbb{R}}\frac{1}{B  |\zeta^{(B_2B_3\alpha_2)}(1 + 2it)|^2} \sum_{v\mid B_2B_3\alpha_2} \frac{1}{v}  \sum_{b_1, b_2 \mid v} \sum_{\gamma_1, \gamma_2\mid B_2B_3\alpha_2/v} \\
&\times \mu(b_1\gamma_1)\mu(b_2\gamma_2)b_1b_2\left(\frac{b_1\gamma_2}{b_2\gamma_1}\right)^{it} \eta\left(\frac{qr}{b_1\gamma_1}, t\right) \eta\left(\frac{B_2n}{b_2\gamma_2}, -t\right)    \mathscr{L}^{\epsilon}\Psi(t) \, \frac{dt}{2\pi}.
\end{split}
\end{equation}
We saw in the previous subsection that the $B_2$-variable was the most critical variable, and we finally used the strong bound \eqref{stein} to get a sufficient saving. We do not have a direct  analogue of this bound in the Eisenstein case, but luckily we can obtain additional cancellation by summing non-trivially over the cusps $v$. This again requires some  subtle manipulations. 

Since $(B_2, B_3\alpha_2) = 1$, we write  $v = v_1v_2$ with $v_1 \mid B_3\alpha_2$, $v_2  \mid B_3\alpha_2$, $b_j = b_j' b_j''$, where $b_j' \mid v_1$, $b_j'' \mid v_2$, and $\gamma_j = \gamma_j'\gamma_j''$, where $\gamma_j' \mid B_2/v_1$, $\gamma_j'' \mid B_3\alpha_2/v_2$. The key observation is that $(qr, B_2) = 1$  in our application, so that   $b_1' = \gamma_1' = 1$. In this way, we can recast the previous $v, b_1, b_2, \gamma_1, \gamma_2$-sum as 
\begin{displaymath}
\begin{split}
\sum_{\substack{b_2' \mid v_1 \mid B_2\\ b_1'', b_2'' \mid v_2 \mid B_3\alpha_2}}\sum_{\gamma_2' \mid \frac{B_2}{v_1}} \sum_{\gamma_1'', \gamma_2'' \mid \frac{B_3\alpha_2}{v_2}} \frac{ \mu(b_1''\gamma_1'') \mu(b_2'b_2'' \gamma_2'\gamma_2'') b_1''b_2'b_2''}{v_1v_2}\left(\frac{b_1''\gamma_2'\gamma_2''}{b_2'b_2''\gamma_1''}\right)^{it}  \eta\left(\frac{qr}{b_1''\gamma_1''}, t\right)\eta\left(\frac{B_2n}{b_2'b_2''\gamma_2'\gamma_2''}, -t\right). 
\end{split}
\end{displaymath}
We consider only the $B_2$-part 
\begin{equation}\label{b2}
\sum_{b_2' \mid v_1 \mid B_2}  \sum_{\gamma_2' \mid \frac{B_2}{v_1}}  \frac{1}{v_1}   \mu(b_2' \gamma_2')  b_2' \left(\frac{\gamma_2'}{b_2'}\right)^{it} \eta\left(\frac{B_2n}{b_2'b_2''\gamma_2'\gamma_2''}, -t\right)
\end{equation}
for fixed $b_2'', \gamma_2''$, where we parametrize
 $v_1 = b_2'b^{\ast}$, $B_2/v_1 = \gamma_2'\gamma^{\ast}$, getting
\[\sum_{b_2'b^{\ast}\gamma_2'\gamma^{\ast} = B_2} \frac{1}{ b^{\ast}} \mu(b_2'\gamma_2') \left(\frac{\gamma_2'}{b_2'}\right)^{it} \eta\left(b^{\ast}\gamma^{\ast} \frac{ n}{ b_2'' \gamma_2''}, -t\right).\]
We must have $b_2''\gamma_2'' \mid n$, so we write $n = b_2''\gamma_2'' n^{\ast}$. Applying the Hecke relation \eqref{mult} for $\eta(n, t)$, we obtain 
\[\sum_{\substack{b_2'b^{\ast}\gamma_2'\gamma^{\ast} = B_2\\ \delta\mid (n^{\ast}, b^{\ast}\gamma^{\ast})}} \mu(\delta) \frac{1}{ b^{\ast}} \mu(b_2'\gamma_2') \left(\frac{\gamma_2'}{b_2'}\right)^{it} \eta \left(\frac{b^{\ast}\gamma^{\ast}}{\delta}, -t\right) \eta \left(\frac{n^{\ast}}{\delta}, -t\right).\]
We parametrize $\delta = \delta_1  \delta_2$, $b^{\ast} = \delta_1b_0$, $\gamma^{\ast} = \delta_2 \gamma_0$, so that the previous line is equal to
\[\sum_{\substack{b_2'\delta_1b_0\gamma_2'\delta_2\gamma_0 = B_2\\ \delta_1\delta_2\mid n^{\ast}}} \mu(\delta_1\delta_2) \frac{1}{ \delta_1b} \mu(b_2'\gamma_2') \left(\frac{\gamma_2'}{b_2'}\right)^{it} \eta  (b_0, -t)\eta(\gamma_0, -t)  \eta \left(\frac{n^{\ast}}{\delta_1\delta_2}, -t\right).\]
The key point is now that by M\"obius inversion, the $b_2', \gamma_2', \gamma_0$-sum disappears, so that \eqref{b2} is equal to
\[\sum_{\substack{ \delta_1b_0 \delta_2 = B_2\\ \delta_1\delta_2\mid n^{\ast}}} \mu(\delta_1\delta_2) \frac{1}{ \delta_1b}   \eta  (b_0, -t)   \eta \left(\frac{n^{\ast}}{\delta_1\delta_2}, -t\right),\]
and hence \eqref{eis-raw} is equal to
\begin{equation*} 
\begin{split}
&  \int_{\mathbb{R}} \underset{b_2\gamma_2 \mid n}{\sum_{  b_1, b_2 \mid v \mid B_3\alpha_2} \sum_{\gamma_1, \gamma_2 \mid  B_3\alpha_2/v}} \frac{ \mu(b_1\gamma_1) \mu( b_2  \gamma_2) b_1 b_2}{ v  B_2B_3\alpha_2 |\zeta^{(B_2B_3\alpha_2)}(1 + 2it)|^2 }\left(\frac{b_1 \gamma_2}{ b_2\gamma_1}\right)^{it}  \eta\left(\frac{qr}{b_1\gamma_1}, t\right) \\
&\times   \sum_{\substack{ \delta_1b_0 \delta_2 = B_2\\ \delta_1\delta_2\mid n}}  \frac{\mu(\delta_1\delta_2)  \eta  (b_0, -t) }{ \delta_1b_0}   \eta \left(\frac{n}{b_2\gamma_2\delta_1\delta_2}, -t\right) 
 \mathscr{L}^{\epsilon}\Psi(t) \, \frac{dt}{2\pi}.
\end{split}
\end{equation*}
After this manoeuvre, we are now in shape to sum over  $r$ and $n$ as in \eqref{final1}. This gives
\begin{equation*} 
\begin{split}
&     \int_{\mathbb{R}}  \sum_{  b_1, b_2 \mid v \mid B_3\alpha_2} \sum_{\gamma_1, \gamma_2 \mid  \frac{B_3\alpha_2}{v}} \frac{ \mu(b_1\gamma_1) \mu( b_2  \gamma_2) b_1 b_2}{ vB_2B_3\alpha_2  |\zeta^{(B_2B_3\alpha_2)}(1 + 2it)|^2 }\left(\frac{b_1 \gamma_2}{ b_2\gamma_1}\right)^{it} \sum_{  \delta_1b_0 \delta_2 = B_2 }  \frac{\mu(\delta_1\delta_2)  \eta  (b_0, -t) }{ \delta_1b_0}   \\
&\times    \sum_{(r,B_2) = 1} \frac{\lambda(rb_1\gamma_1B_3B_4) \eta( qr, t)}{ (b_1\gamma_1r)^{(1-s+u)/2}} \sum_n\frac{ \eta (n, -t) }{(b_2\gamma_2\delta_1\delta_2n)^{(s+u-1+2v)/2}}
 \mathscr{L}^{\pm \sigma\tau}\Psi(t) \, \frac{dt}{2\pi}.
\end{split}
\end{equation*}
The $n$-sum can be easily evaluated in terms of the Riemann zeta function. The $r$-sum requires multiple applications of \eqref{mult}. Checking local factors, one confirms that for $(B, q) = (B_2, Bq) = 1$, $B$ squarefree, and $q$ prime, one has
\[\sum_{(B_2, r) = 1} \frac{\lambda(rB) \eta(qr, t)}{r^z} =  \frac{\eta(q, t) - \lambda(q)q^{-z}}{1 - q^{-2z}}  \prod_{p \mid B} \frac{\lambda(p) - \eta(p, t)p^{-z}}{1 - p^{-2z} }\sum_{(B_2, r) = 1} \frac{\lambda(r) \eta(r, t)}{r^z}.\]
 Putting everything together, the Eisenstein contribution to \eqref{final1} is equal to
\begin{equation}\label{eiscont}
\begin{split}
  \int_{\mathbb{R}}  \Theta^{\textnormal{Eis}}_{a, b, q}(s, u, v,t)   \frac{\zeta(\frac{s+u-1+2v}{2} + it)L(\frac{s+u-1+2v}{2} - it) \zeta( \frac{1-s+u}{2} + it, f )L( \frac{1-s+u}{2} - it, f)}{\zeta(1+ 2it) \zeta(1 - 2it)} \, \frac{dt}{2\pi},
\end{split}
\end{equation}
where
 \begin{equation}\label{defthetaeis}
\begin{split}
&\Theta^{\textnormal{Eis}}_{a, b, q}(s, u, v,   t) \defeq  \sum_{\pm} \sum_{\sigma, \tau  \in\{ \pm\}}  \sum_{\alpha_1\alpha_2 = a} \sum_{\beta_1B_3B_4B_2=  b} \frac{q^{(3-s-u-2v)/2}  \mu(B_2)\mu(B_3)\alpha_1^{s-1}\alpha_2^{u} }{\beta_1^u B_2^{(s-u-1-2v)/2}  B_3^{1-v } B_4^{1-v} } \\ 
&   \times    \sum_{  b_1, b_2 \mid v \mid B_3\alpha_2} \sum_{\gamma_1, \gamma_2 \mid  \frac{B_3\alpha_2}{v}} \frac{ \mu(b_1\gamma_1) \mu( b_2  \gamma_2) b_1 b_2}{ vB_2B_3\alpha_2  |\zeta^{(B_2B_3\alpha_2)}(1 + 2it)|^2 }\left(\frac{b_1 \gamma_2}{ b_2\gamma_1}\right)^{it}\frac{\eta(q, t) - \lambda(q)q^{-(1-s+u)/2}}{1 - q^{-(1-s+u)}} \\
& \times \sum_{  \delta_1b_0 \delta_2 = B_2 }  \frac{\mu(\delta_1\delta_2)  \eta  (b_0, -t) }{ \delta_1b_0  (b_1\gamma_1)^{(1-s+u)/2} (b_2\gamma_2\delta_1\delta_2)^{(s+u-1+2v)/2}} \prod_{p \mid b_1\gamma_1B_3B_4} \frac{\lambda(p) - \eta(p, t)p^{-(1-s+u)/2}}{1 - p^{-(1-s+u)} }   \\
&\times  \frac{\zeta_{B_2B_3\alpha_2}(1+2it)\zeta_{B_2B_3\alpha_2}(1-2it)}{\zeta_{\alpha_2 B_3B_4q}(1-s+u)L_{B_2}( \frac{1-s+u}{2} + it, f )L_{B_2}( \frac{1-s+u}{2} - it, f)} \mathscr{L}^{\pm\sigma \tau}\Psi^{\pm, \sigma, \tau}_{a, b, s, u, v}(t).
\end{split}
\end{equation}
  The term \eqref{eiscont} is clearly holomorphic in the range \eqref{range} and it can easily be extended as long as $\Re (u-s) > 1$ and $\Re(s+u+2v) > 3$. To pass these two hyperplanes, we observe that the presence of the Riemann zeta function in the numerator contributes residues, and so we apply the argument of \cite[Lemma 16]{BK1}   to show that the meromorphic continuation of \eqref{eiscont} in the region $\Re (u-s) < 1$ and $\Re(s+u+2v) < 3$ is given by the same expression plus the polar term
\begin{equation}\label{P3}
\begin{split}
\mathcal{P}^{(3)}_{a, b, q}(s, u, v) \defeq \sum &\Res_{\substack{t = \pm \frac{i}{2}(1+s-u)\\ t = \pm \frac{i}{2}(3-s-u-2v)}}  (\pm i)   \frac{\Theta^{\textnormal{Eis}}_{a, b, q}(s, u, v, t)\zeta(\frac{s+u-1+2v}{2} + it) }{\zeta(1+ 2it) \zeta(1 - 2it)}  \\
&\times \zeta\Big(\frac{s+u-1+2v}{2} - it\Big) L\Big( \frac{1-s+u}{2} + it, f  \Big)L\Big( \frac{1-s+u}{2} - it, f  \Big) .
\end{split}
\end{equation}
A trivial estimation confirms  \eqref{boundthetaeis} for the term on the right-hand side of \eqref{eiscont} with $\Re s = \Re u = \Re v = 1/2$, $t \in \mathbb{R}$, $a\asymp b$ (which differs from the meromorphic continuation of \eqref{eiscont} to this region by \eqref{P3}). 

It remains to meromorphically continue and bound the joint polar term
\begin{equation}\label{Pall}
\mathcal{P}_{a, b, q}(s, u, v) \defeq \sum_{j=1}^3\mathcal{P}^{(j)}_{a, b, q}(s, u, v),
\end{equation}
where we recall \eqref{P1} and \eqref{P2} for $j=1, 2$. In these cases, it is easily seen that $\mathcal{P}^{(j)}_{a, b, q}(s, u, v)$ continues meromorphically to a neighbourhood of \eqref{larger}, and for $1/2 - \varepsilon < \Re s= \Re u = \Re v < 1/2+\varepsilon$, $a\asymp b$, we have the bound
\[|\mathcal{P}^{(1)}_{a, b, q}(s, u, v)| + |\mathcal{P}^{(2)}_{a, b, q}(s, u, v)| \ll  q (ab)^{-1/2} (abq)^{\varepsilon}\]
away from poles. The treatment of $\mathcal{P}^{(3)}_{a, b, q}(s, u, v)$ requires slightly more effort, because we need to analyze $\Theta^{\textnormal{Eis}}_{a, b, q}(s, u, v,  t)$ for $|\Im t| \leq 1/2$. The meromorphic continuation of $ \mathscr{L}^{\pm\sigma \tau}\Psi^{\pm, \sigma, \tau}_{a, b, s, u, v}(t)
$ with at most finitely many poles (and hence of   $\Theta^{\textnormal{Eis}}_{a, b, q}(s, u, v,   t)$) to that region 
follows from \cite[Lemma 3b]{BK1}. Again, a trivial upper bound yields
\[\Theta^{\textnormal{Eis}}_{a, b, q}(s, u, v,   t) \ll (abq)^{\varepsilon} \frac{q}{(ab)^{1/2-\theta}}\]
for fixed $s, u, v, t$ with 
$1/2 - \varepsilon < \Re s= \Re u = \Re v < 1/2+\varepsilon$, $|\Im t| < 1/2+\varepsilon$, $a\asymp b$ away from possible poles, so that  also 
\[\mathcal{P}^{(3)}_{a, b, q}(s, u, v) \ll  (abq)^{\varepsilon} \frac{q}{(ab)^{1/2-\theta}}\]
in the region $1/2 - \varepsilon < \Re s= \Re u = \Re v < 1/2+\varepsilon$, away from possible poles. We have established \eqref{formula} as an equality of meromorphic functions, but since all terms except possibly   $\mathcal{P}_{a, b, q}(s, u, v)$ are holomorphic for $\Re s = \Re u = \Re v = 1/2$, $\mathcal{P}_{a, b, q}(s, u, v)$ must also be holomorphic for $\Re s = \Re u = \Re v = 1/2$, and the general bound \eqref{boundP} then follows by Cauchy's integral theorem in the same way as at the end of \cite[Section 10]{BK1}.  

\section{Proof of Theorem \ref{thm2}}

\subsection{Initial manipulations} Let $P = TQ$. By ``negligible'', we mean a quantity that is $O(P^{-100})$.  By a dyadic decomposition, we may replace the conditions $q \leq Q$, $|t_{\psi}| \leq T$ with $\frac{1}{2}Q \leq q \leq  Q$, $\frac{1}{2}T \leq t_{\psi} \leq T$ or $t_{\psi} \in [0, 1]\cup [-i\theta, i\theta]$ where in the last case we formally put $T=1$.  

Let $E_3$ denote the standard minimal Eisenstein series for $\mathrm{SL}_3(\mathbb{Z})$ with Fourier coefficients
\[A(n, m)  = \sum_{d\mid (n, m)} \mu(d) \tau_3(n/d)\tau_3(m/d).\]
Then for $\psi \in \mathcal{B}^{\ast}(q)$ with $|t_{\psi}| \leq T$, we have
\[L(s, \psi)^3 = L(s, \psi \times E_3) = \sum_{n}\sum_{(m, q) = 1} \frac{A(n, m) \lambda_{\psi}(n)}{n^s m^{2s}} =  P_q(s) \sum_{n, m} \frac{A(n, m) \lambda_{\psi}(n)}{n^s m^{2s}}\]
for $\Re s > 1$, where
\[P_q(s) = \prod_{p \mid q} \left(1 - \frac{\lambda_{\psi}(p)}{p^s}\right)^{-3}\left(1 - \frac{\lambda_{\psi}(p)}{p^s} + \frac{1}{p^{2s}}\right)^{3}\]
is holomorphic and uniformly bounded in $\Re s \geq 1/2$. By a standard approximate functional equation,   we have
\[|L(1/2, \psi)^3| \leq 2 \Big|\sum_{n, m} \frac{A(n, m) \lambda_{\psi}(n)}{n^{1/2} m} V_{\psi}\left(\frac{nm^2}{q^{3/2}}\right) \Big|,\]
where
\[V_{\psi}(y) =\frac{1}{2\pi i} \int_{(2)} P(1/2 + u) \frac{\Gamma(\frac{1}{2}(\frac{1}{2} + u +\epsilon_{\psi}+ it_{\psi}))^3\Gamma(\frac{1}{2}(\frac{1}{2} + u+\epsilon_{\psi} - it_{\psi}))^3}{ \Gamma(\frac{1}{2}(\frac{1}{2}   +\epsilon_{\psi}+ it_{\psi}))^3\Gamma(\frac{1}{2}(\frac{1}{2}+\epsilon_{\psi}  - it_{\psi}))^3}  \pi^{-3u} e^{u^2} y^{-u} \, \frac{du}{u}.\]
Shifting the contour to the far right, we see that $V_{\psi}(y)$ is negligible if $y \geq { T}^3 P^{\varepsilon}$. Remembering this, we shift the contour to $\Re u  =  \varepsilon$. There we may truncate the integral at $|\Im u|\leq P^{\varepsilon}$ at the cost of a negligible error. Applying a smooth dyadic decomposition, we have shown
\[L(1/2, \psi)^3 \ll_{\varepsilon} P^{\varepsilon} \int_{-P^{\varepsilon}}^{P^{\varepsilon}} \sum_{2^\nu = {N}  \leq  {  Q}^{3/2}{ T}^3P^{\varepsilon}}   \left|\sum_{n, m} \frac{A(n, m) \lambda_{\psi}(n)}{(nm^2)^{1/2+ iv} } V\left(\frac{nm^2}{{ N}}\right)\right| \, dv,\]
where $V$ has support in $[1, 2]$, is independent of $\psi$, and satisfies $V^{(j)}(y) \ll_j  1$ for all $j\in \mathbb{N}_0$. Multiplying two such expressions together and using the Cauchy--Schwarz inequality, we obtain
\[L(1/2, \psi)^6 \ll_{\varepsilon} P^{\varepsilon}\max_{|v| \leq P^{\varepsilon}} \max_{ {  N}  \leq  {  Q}^{3/2}{  T}^3P^{\varepsilon}}   \sum_{ n_1, n_2, m_1, m_2 } \frac{A(n_1, m_1)A(n_2, m_2) \lambda_{\psi}(n_1)\lambda_{\psi}(n_2)}{(n_1m_1^2)^{1/2+ iv}(n_2  m_2^2)^{1/2-iv}} V\left(\frac{n_1m_1^2}{{  N}}\right)\overline{V\left(\frac{n_2m_2^2}{{  N}}\right)}.\]
For $\psi \in \mathcal{B}^{\ast}(q)$, we have
\[\frac{\lambda_{\psi}(n)\lambda_{\psi}(m)}{L(1, \mathrm{Ad}^2\psi)}  =  q  \prod_{p\mid q}\left(1 - \frac{1}{p}\right)^{-1} \rho_{\psi, 1, q}(n) \rho_{\psi, 1, q}(m)\]
by \eqref{rho-cusp}.  
For the purpose of Theorem \ref{thm2}, it therefore suffices to bound
\begin{equation}\label{sufficient}
\begin{split}
\mathcal{S}_v({  Q}, {  T}, {  N}) \defeq  &  \sum_{q}  W\left(\frac{q}{{  Q}}\right)  {  Q} \sum_{\psi \in \mathcal{B}^{\ast}(q)}  h_{  T}(t_{\psi}) \\
&\times \sum_{ n_1, n_2, m_1, m_2 } \frac{A(n_1, m_1)A(n_2, m_2) \rho_{\psi, 1, q}(n_1)\rho_{\psi, 1, q}(n_2)}{(n_1m_1^2)^{1/2+ iv}(n_2  m_2^2)^{1/2-iv}}V\left(\frac{n_1m_1^2}{{  N}}\right)\overline{V\left(\frac{n_2m_2^2}{{  N}}\right)},
\end{split}
\end{equation}
where   ${  N} \leq {  Q}^{3/2} {  T}^3 P^{\varepsilon}$, $|v| \leq P^{\varepsilon}$, and 
\[h_{  T}(t) =  e^{-(t/T)^2} \prod_{n=1}^{\lfloor \varepsilon^{-1} \rfloor}\left(\frac{1}{{  T}^2} \Big(t^2 + \frac{(2n-1)^2}{4}\Big)\right).\]
Note that this function satisfies the assumptions of Lemmas \ref{simple-K} and \ref{hard-K}. 



\subsection{The Eisenstein contribution associated with the trivial character}

The $\psi$-sum in \eqref{sufficient} can be evaluated by the Kuznetsov formula \eqref{kuz-con}. To this end, we need to add, using positivity, the contribution from the oldforms and the continuous spectrum. As mentioned in the introduction, this manoeuvre is costly, and we single out the contribution of the continuous spectrum associated with the trivial character:
\begin{displaymath}
\begin{split}
\mathcal{S}^{\ast}_v({  Q}, &{  T}, {  N})  \defeq     \sum_{q}  W\left(\frac{q}{{  Q}}\right)  {  Q} \sum_{M \mid q} \int_{\mathbb{R}}  h_{  T}(t) \\
&\times\sum_{ n_1, n_2, m_1, m_2 } \frac{A(n_1, m_1)A(n_2, m_2) \rho_{\textnormal{triv}, M, q}(n_1, t)\overline{\rho_{\textnormal{triv},  M, q}(n_2, t)}}{(n_1m_1^2)^{1/2+ iv}(n_2  m_2^2)^{1/2-iv}}V\left(\frac{n_1m_1^2}{{  N}}\right)\overline{V\left(\frac{n_2m_2^2}{{  N}}\right)} \, \frac{dt}{2\pi },
\end{split}
\end{displaymath}
which we re-write  in  more compact form as
\begin{equation}\label{compact}
\begin{split}
     \int_{(2)} \int_{(2)}\int_{(1)}  \int_{\mathbb{R}}& \widehat{V}(z_1)\widehat{V}(z_2) \widehat{W}(s) {  Q}^{1+s} {  N}^{z_1+z_2} \\
   &  \times \frac{\mathcal{D}_t(s, 1/2 + iv + z_1, 1/2 - iv + z_2; 0, 0)}{|\zeta(1 + 2it)|^2}  h_{  T}(t)   \, \frac{dt}{2\pi } \, \frac{ds\, dz_1\, dz_2}{(2\pi i)^3},
  \end{split}
\end{equation}
where
\[\mathcal{D}_t({\tt s},  {\tt z}_1,   {\tt z}_2; {\tt w}_1, {\tt w}_2) = |\zeta(1 + 2it) |^2\sum_{q, n_1,   n_2, m_1, m_2}\sum_{M \mid q}  \frac{A(n_1, m_1)A(n_2, m_2) \rho_{\textnormal{triv}, M, q}(n_1, t)\overline{\rho_{\textnormal{triv},  M, q}(n_2, t)}}{q^{\tt s} (n_1m_1^2)^{{\tt z}_1}(n_2m_2^2)^{{\tt z}_2} m_1^{2{\tt w}_1}m_2^{2{\tt w}_2}}.\]
Recalling the definition (cf.\ \eqref{rho-eis1})
 \begin{displaymath}
\begin{split}
&|\zeta(1 + 2it) |^2 \rho_{\textnormal{triv}, M, q}(n_1, t)\overline{\rho_{\textnormal{triv}, M, q}(n_2, t)} \\
&= \frac{|\zeta^{(q)}(1 + 2it) |^2 (n_1/n_2)^{it}}{q \nu(q)\tilde{\mathfrak{n}}_q(M)^2  }  \sum_{\delta_1, \delta_2 \mid M} \frac{\delta_1\delta_2 \mu(M/\delta_1) \mu(M/\delta_2)}{M} \sum_{\substack{c_1\delta_1 f_1 = n_1\\ (c_1, q/M) = 1}} \sum_{\substack{c_2\delta_2 f_2 = n_2\\ (c_2, q/M) = 1}} \left(\frac{c_2}{c_1}\right)^{2it},
\end{split}
\end{displaymath}
we see that (for $t\in \mathbb{R}$) the series $ \mathcal{D}_t({\tt s},  {\tt z}_1,   {\tt z}_2; {\tt w}_1, {\tt w}_2) $ is absolutely convergent in $\Re {\tt s}> 0$, $\Re {\tt z}_1, \Re {\tt z}_2 > 1$, $\Re({\tt z}_1+ {\tt w}_1), \Re({\tt z}_2 +  {\tt w}_2) > 1/2$, and 
 admits an Euler product of the shape
\[\prod_p \left(1 + \frac{1}{p^{{\tt s}+1}} + \sum_{j=1, 2}\sum_{\pm}\frac{3}{p^{{\tt z}_j \pm it}}  + O\left(\frac{1}{p^{2\min(\Re {\tt z}_1, \Re {\tt z}_2, \Re({\tt z}_1 + {\tt w}_1), \Re ({\tt z}_2+{\tt w}_2))}} + \frac{1}{p^{\Re{\tt s} + \min(\Re {\tt z}_1, 1) + \min(\Re {\tt z}_2, 1)}}\right) \right),\]
where the bounds in the error term hold uniformly in 
\[\Re {\tt z}_1, \quad \Re {\tt z}_2, \quad \Re({\tt z}_1 + {\tt w}_1),\quad  \Re ({\tt z}_2+{\tt w}_2)), \quad \Re{\tt s} + \min(\Re {\tt z}_1, 1) + \min(\Re {\tt z}_2, 1) > 0.\]
 In particular, we have 
\begin{equation}\label{D}
 \mathcal{D}_t({\tt s},  {\tt z}_1,   {\tt z}_2; {\tt w}_1, {\tt w}_2)  = \zeta({\tt s}+1)  \zeta({\tt z}_1 + it)^3 \zeta({\tt z}_1 - it)^3 \zeta({\tt z}_2+ it)^3 \zeta({\tt z}_2-it)^3 \mathcal{E}_t({\tt s},  {\tt z}_1,   {\tt z}_2; {\tt w}_1, {\tt w}_2),
 \end{equation}
where $\mathcal{E}_t({\tt s},  {\tt z}_1,   {\tt z}_2; {\tt w}_1, {\tt w}_2)$ is holomorphic and uniformly bounded in 
\begin{equation}\label{regionE}
  \Re {\tt z}_1,   \Re {\tt z}_2,   \Re({\tt z}_1 + {\tt w}_1),   \Re ({\tt z}_2+{\tt w}_2)) \geq 1/2+\varepsilon, \quad \Re{\tt s} + \min(\Re {\tt z}_1, 1) + \min(\Re {\tt z}_2, 1)   \geq 1+\varepsilon
  \end{equation}
  as long as $\Im t = 0$. 
Hence in \eqref{compact}, we may shift the contours to $\Re s = -1+\varepsilon$ (picking up a residue at $s = 0$), and in the remaining integral we  shift the $z_1, z_2$-contours  to $\Re z_1 = \Re z_2 = 1/2 + \varepsilon$, getting
\begin{equation}\label{eis-final}
\begin{split}
\mathcal{S}^{\ast}_v&({  Q}, {  T}, {  N})  =     \widehat{W}(0){  Q}    \int_{(2)} \int_{(2)}   \int_{\mathbb{R}} \prod_{\pm}\frac{  \zeta(1/2 + iv + z_1 \pm it)^3 \zeta(1/2 - iv + z_2\pm it)^3 }{\zeta(1 \pm 2it)}  \\
& \times   \mathcal{E}_t(0, 1/2 + iv + z_1, 1/2 - iv + z_2; 0, 0)\widehat{V}(z_1)\widehat{V}(z_2)     {  N}^{z_1+z_2}
 h_{  T}(t)  \,  \frac{  dz_1\, dz_2}{(2\pi i)^2}  \, \frac{dt}{2\pi }  +  O({  T} {  N} P^{\varepsilon}).
\end{split}
\end{equation}

\subsection{Applying the Kuznetsov formula twice} 
By the Kuznetsov formula (and positivity), we obtain
 \begin{displaymath}
\begin{split}
\mathcal{S}_v({  Q}, {  T}, {  N})  +\mathcal{S}^{\ast}_v({  Q}, {  T}, {  N})  \leq  & {  Q} \sum_{q}  W\left(\frac{q}{{  Q}}\right)\sum_{ n_1, n_2, m_1, m_2 } \frac{A(n_1, m_1)A(n_2, m_2)  }{(n_1m_1^2)^{1/2+ iv}(n_2  m_2^2)^{1/2-iv}}V\left(\frac{n_1m_1^2}{{  N}}\right)\overline{V\left(\frac{n_2m_2^2}{{  N}}\right)}  \\
&  \times  \Bigg(\delta_{n_1, n_2} \int_{-\infty}^{\infty} h_{  T}(t) \frac{t\tanh(\pi t) \, dt}{2\pi^2} + \sum_{  c} \frac{S(n_1, n_2, qc)}{qc} \mathscr{K}h_{  T}\left(\frac{\sqrt{n_1n_2}}{qc}\right)\Bigg) .
\end{split}
\end{displaymath}
The diagonal term is easy to deal with and is trivially bounded by 
\begin{equation}\label{diag}
O_{\varepsilon}\left(P^{\varepsilon}   {  Q}^2 {  T}^2\right).
\end{equation}
 By Mellin inversion, we can recast the off-diagonal term as
 \begin{displaymath}
\begin{split}
    \int_{(-12)}  {  Q}^{s+1} \widehat{W}(s)  & \sum_{ n_1, n_2, m_1, m_2 } \frac{A(n_1, m_1)A(n_2, m_2)  }{(n_1m_1^2)^{1/2+ iv }(n_2  m_2^2)^{1/2-iv } (n_1n_2)^{\frac{s}{2}}} V\left(\frac{n_1m_1^2}{{  N}}\right)\overline{V\left(\frac{n_2m_2^2}{{  N}}\right)} \\
    &\times \sum_{ c, q} c^s \frac{S(n_1, n_2, qc)}{qc} \mathscr{K}_sh_{  T}\left(\frac{\sqrt{n_1n_2}}{qc}\right) \, \frac{ds}{2\pi i} 
\end{split}
\end{displaymath}
with 
$\mathscr{K}_sh_{  T}(x) = x^s \mathscr{K}h_{  T}(x) $ as in Section \ref{26}. Applying the Kuznetsov formula immediately in the other direction (which we may do by Lemma \ref{simple-K}), we obtain  by \eqref{kuz2} that the previous expression is equal to
 \begin{equation}\label{after-kuz}
\begin{split}
&  \int_{(-12)}  {  Q}^{s+1} \widehat{W}(s)   \sum_{ n_1, n_2, m_1, m_2 } \frac{A(n_1, m_1)A(n_2, m_2)  }{(n_1m_1^2)^{1/2+ iv }(n_2  m_2^2)^{1/2-iv  }(n_1n_2)^{\frac{s}{2}}}V\left(\frac{n_1m_1^2}{{  N}}\right)\overline{V\left(\frac{n_2m_2^2}{{  N}}\right)}\\
  & \times  \sum_{ c} c^s \Bigl(\mathcal{A}^{\textnormal{Maa{\ss}}}_c(n_1, n_2; \mathscr{L}^+\mathscr{K}_sh_{  T})+\mathcal{A}^{\textnormal{Eis}}_c(n_1, n_2; \mathscr{L}^+\mathscr{K}_sh_{  T})  + \mathcal{A}^{\textnormal{hol}}_c(n_1, n_2; \mathscr{L}^{\textnormal{hol}}\mathscr{K}_sh_{  T})\Bigr) \, \frac{ds}{2\pi i}.
\end{split}
\end{equation}
Lemma \ref{hard-K}b) implies that  $\mathscr{L}^+\mathscr{K}_sh_{  T}(t)$ has analytic continuation to $\Re s < 1$, and we proceed to derive a uniform bound. If $|t| \geq 10 |\Im s|$ (so that $t \pm \frac{1}{2}|\Im s| \asymp t$), we have
\begin{displaymath}
\begin{split}
\mathscr{L}^+\mathscr{K}_sh_{  T}(t) & \ll  \int_{\mathbb{R}} \frac{e^{-|\tau|/T} (1 + |\tau|)}{ (1+ |t| + |\tau|)^{2-2\Re s}} \, d\tau +  \frac{e^{-|t|/T}(1+|t|)}{(1 + |t|)^{1-\Re s} } \int_0^{1+|t|} \frac{1}{(1+|\tau|)^{1-\Re s}} \, d\tau\\
& \ll \frac{T^2}{(1+|t|)^{2-2\Re s }} +  e^{-|t|/T} \left( (1 + |t|)^{\Re s} + (1 + |t|)^{2\Re s} \right) \ll \frac{T^{2+\max(0, -\Re s)}}{(1+|t|)^{2-2\Re s }}  . 
\end{split}
\end{displaymath}
If $|t| \leq 10 |\Im s|$, we have trivially $\mathscr{L}^+\mathscr{K}_sh_{  T}(t) \ll T^2$, so that altogether we obtain the uniform bound
\begin{equation}\label{uniform}
\mathscr{L}^+\mathscr{K}_sh_{  T}(t)   \ll_{\Re s} (1+|\Im s|)^{2-2\Re s} \frac{T^{2+\max(0, -\Re s)}}{(1+|t|)^{2-2\Re s }}  . 
\end{equation}

The problematic expression in \eqref{after-kuz} is the part of $\mathcal{A}^{\textnormal{Eis}}_c(n_1, n_2; \mathscr{L}^+\mathscr{K}_sh_{  T})$ that is associated with the trivial character. We spell this out explicitly as
\begin{displaymath}
\begin{split}
\mathcal{S}_v^{\ast\ast}({  Q}, {  T}, {  N}) & =    \int_{(20)} \int_{(20)}\int_{(-12)}  \int_{\mathbb{R}} \widehat{V}(z_1)\widehat{V}(z_2) \widehat{W}(s) {  Q}^{1+s} {  N}^{z_1+z_2} \\
&\times \frac{\mathcal{D}_t\left(-s, \frac{1}{2} + iv  + \frac{s}{2}+ z_1, \frac{1}{2} - iv + \frac{s}{2}+ z_2; - \frac{s}{2}, -\frac{s}{2}\right) }{|\zeta(1 + 2it)|^2}  \mathscr{L}^+\mathscr{K}_sh_{  T}(t )  \, \frac{dt}{2\pi }  \, \frac{ds\, dz_1\, dz_2}{(2\pi i)^3}. 
\end{split}
\end{displaymath}

Shifting the $s$-contour to the far left and simultaneously the $z_1, z_2$-contours to $\Re z_1 = \frac{1}{2}(1  - \Re s)+ \varepsilon$, we see from \eqref{uniform} that the $t$-integral is negligible for $|t| \geq \sqrt{NT/Q}P^{\varepsilon}$. In particular, we may truncate at $|t| \leq (T + \sqrt{NT/Q})P^{\varepsilon}$. 

Next we shift the $s$-contour to $\Re s = \varepsilon$, past the pole at $s=0$. By Lemma \ref{hard-K}a),  the residue  matches exactly the main term in \eqref{eis-final} except for the truncation of the $t$-integral, but by the rapid decay of $  \mathscr{L}^+\mathscr{K}_0h_{  T} = h_{T}$ for $|t| \geq T$, we may re-insert the tail at the cost of a negligible error. 

To estimate the remaining integral,  we shift the $z_1, z_2$-contours to left, past the triple poles at $z_1  = 1/2 - iv -\frac{s}{2} \pm it$, $z_2  = 1/2 - iv -\frac{s}{2} \pm it$ to $\Re z_1, \Re z_2 = \varepsilon$. Thus we need to bound the contributions from the remaining integral and the two residues. The remaining multiple integral contains a $t$-integral that can be bounded by 
\[\ll_{\varepsilon} \int_{|t| \leq  (T + \sqrt{NT/Q})P^{\varepsilon} } |\zeta(\textstyle\frac{1}{2} + \varepsilon + it + i\tau)|^{12} \displaystyle  \frac{T^2}{(1+|t|)^{2-\varepsilon}}  \, dt \ll_{\varepsilon}   T^2 (1+|\tau|)^2  P^{\varepsilon},\]
where $\tau = \pm v + \Im z_j + \frac{1}{2} \Im s_j$ and we used Heath-Brown's twelfth moment bound \cite{HB}.   Thus the total contribution of the remaining integral is $O_{\varepsilon}(Q T^2 P^{\varepsilon})$. It remains to deal with the two residues. Here the rapid decay of $\widehat{W}$ and $\widehat{V}_{1, 2}$ and their derivatives at $z  = 1/2 \pm iv -\frac{s}{2} \pm it$  makes the $t$-integral rapidly convergent regardless of the real part of $s$, so we may shift the contour to $\Re s = 1-\varepsilon$   (so that $\Re z_j = \varepsilon$), getting a contribution of $O_{\varepsilon}(Q^2 T^2 P^{\varepsilon})$. 

Combining \eqref{diag} and the  error term in \eqref{eis-final} with the previous two error terms, we have accomplished so far the bound
 \begin{equation}\label{sofar}
\begin{split}
\mathcal{S}_v&({  Q}, {  T}, {  N})   \ll_{\varepsilon}   P^{\varepsilon} (Q^2 T^2 + NT)  \\
&+\Bigl| \int_{(-12)}  {  Q}^{s+1} \widehat{W}(s)   \sum_{ n_1, n_2, m_1, m_2 } \frac{A(n_1, m_1)A(n_2, m_2)  }{(n_1m_1^2)^{1/2+ iv }(n_2  m_2^2)^{1/2-iv  }(n_1n_2)^{\frac{s}{2}}}V\left(\frac{n_1m_1^2}{{  N}}\right)\overline{V\left(\frac{n_2m_2^2}{{  N}}\right)}\\
  & \times  \sum_{ c} c^s \Bigl(\mathcal{A}^{\textnormal{Maa{\ss}}}_c(n_1, n_2; \mathscr{L}^+\mathscr{K}_sh_{  T})+\mathcal{A}^{\textnormal{Eis}, \ast}_c(n_1, n_2; \mathscr{L}^+\mathscr{K}_sh_{  T})  + \mathcal{A}^{\textnormal{hol}}_c(n_1, n_2; \mathscr{L}^{\textnormal{hol}}\mathscr{K}_sh_{  T})\Bigr) \, \frac{ds}{2\pi i}\Bigr|,
\end{split}
\end{equation}
where $\mathcal{A}^{\textnormal{Eis}, \ast}_c$ denotes the contribution of level $c$ Eisenstein series without the trivial character.

\subsection{The endgame} We consider the Maa{\ss} contribution in \eqref{sofar} given by
 \begin{equation*}
\begin{split}
  \int_{(-12)} & { Q}^{s+1} \widehat{W}(s)  \sum_{ c} c^s  \sum_{c_0M  \mid c}   \sum_{\psi \in \mathcal{B}^{\ast}(c_0)}   \sum_{ n_1, n_2, m_1, m_2 } \frac{A(n_1, m_1)A(n_2, m_2) \rho_{\psi, M, N}(n) \overline{ \rho_{\psi, M, N}(m)}  }{(n_1m_1^2)^{1/2+ iv  }(n_2  m_2^2)^{1/2-iv  }(n_1n_2)^{\frac{s}{2}}}\\
  &\times V\left(\frac{n_1m_1^2}{{ N}}\right)\overline{V\left(\frac{n_2m_2^2}{{ N}}\right)}   \mathscr{L}^+\mathscr{K}_sh_{\tt T}(t_{\psi})   \, \frac{ds}{2\pi i}.
\end{split}
\end{equation*}
Shifting the $s$-contour to the far left, we see that we can truncate both the $c$-sum and the $\psi$-sum at $c(1 + |t_{\psi}|^2) \leq P^{\varepsilon}NT/Q$  at the cost of a negligible error (recall \eqref{uniform} and the rapid decay of $\widehat{V}$). Having done this, we shift the $s$-contour back to $\Re s = 0$. By Mellin inversion, we obtain
 \begin{equation}\label{almostdone}
\begin{split}
  \int_{(0)} & Q^{s+1} \widehat{W}(s) \underset{   c(1 + |t_{\psi}|^2) \leq P^{\varepsilon} NT/Q  }{ \sum_{ c} \sum_{c_0 \mid c}   \sum_{\psi \in \mathcal{B}^{\ast}(c_0)}} \frac{c^s}{c\nu(c)} \prod_{p \mid c_0} (1 - p^{-2}) \int_{( \varepsilon)}\int_{( \varepsilon)} N^{z_1+z_2} \widehat{V}(z_1)\widehat{V}(z_2)    \\
  &\times \frac{\mathcal{D}^{\textnormal{Maa{\ss}}}_{\psi, c}(\frac{1}{2} + iv + \frac{s}{2} + z_1, \frac{1}{2} - iv + \frac{s}{2} + z_2, -\frac{s}{2}, - \frac{s}{2} ) }{ L(1, \mathrm{Ad}^2 \psi)} \mathscr{L}^+\mathscr{K}_sh_{\tt T}(t_{\psi})  \, \frac{dz_1 \, dz_2}{(2\pi i)^2} \, \frac{ds}{2\pi i},
\end{split}
\end{equation}
where (recalling the notation in \eqref{rho-cusp}) 
\[\mathcal{D}^{\textnormal{Maa{\ss}}}_{\psi, c}( {\tt z}_1,  {\tt  z}_2, {\tt w}_1, {\tt w}_2 )  = \frac{L(1, \mathrm{Ad}^2 \psi) c\nu(c)}{ \prod_{p \mid c_0} (1 - p^{-2})} \sum_{M \mid \frac{c}{c_0}} \sum_{ n_1, n_2, m_1, m_2 } \frac{A(n_1, m_1)A(n_2, m_2) \rho_{\psi, M, c}(n) \overline{ \rho_{\psi, M, c}(m)}  }{(n_1m_1^2)^{{\tt z}_1  }(n_2  m_2^2)^{{\tt z}_2  } m_1^{2{\tt w}_1} m_2^{2{\tt w}_2}}\]
for $\psi \in \mathcal{B}^{\ast}(c_0)$ with $c_0 \mid c$. 
Using \eqref{rho-cusp} and \eqref{xi-arithmetic} (with $\theta \leq 1/2$), we see as in \eqref{inparticular}   that
\begin{equation}\label{DMaass}
\mathcal{D}^{\textnormal{Maa{\ss}}}_{\psi, c}( {\tt z}_1,  {\tt  z}_2, {\tt w}_1, {\tt w}_2 ) = L({\tt z}_1, \psi)^3L({\tt z}_2, \psi)^3 \mathcal{E}_{\psi, c}^{\textnormal{Maa{\ss}}}( {\tt z}_1,  {\tt  z}_2, {\tt w}_1, {\tt w}_2 ),
\end{equation}
where 
\[\mathcal{E}_{\psi, c}^{\textnormal{Maa{\ss}}}( {\tt z}_1,  {\tt  z}_2, {\tt w}_1, {\tt w}_2 )\ll_{\varepsilon} c^{\varepsilon}\]
uniformly in $\Re {\tt z}_1, \Re {\tt z}_2, \Re ({\tt z}_1 + {\tt w}_1), \Re ({\tt z}_2+ {\tt w}_2) \geq 1/2 + \varepsilon$. The convexity bound for $L({\tt z}, \psi)$ is 
\[L({\tt z}, \psi) \ll_{\varepsilon}  \big(c_0(1 + |t_{\psi}| + |\Im {\tt z}|)^2\big)^{1/4+\varepsilon}, \quad \Re {\tt z} \geq 1/2.\]
We can afford to use the convexity bound on four of the six $L$-functions in \eqref{DMaass}. We may then  truncate the $s, z_1, z_2$-contours at height $P^{\varepsilon}$,  and after a trivial estimation, we bound \eqref{almostdone} by
\begin{equation}\label{maass-final}
\ll_{\varepsilon} P^{\varepsilon} Q \, \frac{NT}{Q}\max_{|\xi| \leq P^{\varepsilon}} \underset{   c(1 + |t_{\psi}|^2) \leq P^{\varepsilon} NT/Q  }{ \sum_{ c}  \sum_{\psi \in \mathcal{B}^{\ast}(c)}} \frac{1}{c}  |L(\textstyle \frac{1}{2} +\varepsilon + i\xi)|^2  \displaystyle \frac{T^2}{(1 + |t_{\psi}|^2)}.
\end{equation}
It is an easy exercise with the Kuznetsov formula or the spectral large sieve to obtain a Lindel\"of on average bound for the second moment, which can safely be left to the reader:  the length of the approximate functional equation in each factor is $O_{\varepsilon}(P^{\varepsilon} c^{1/2}(1 + |t_{\psi}|))$, so the Kloosterman term in the Kuznetsov formula is essentially invisible. Thus by Weyl's law, the total contribution of the previous expression is
\[\ll_{\varepsilon} P^{\varepsilon} N^2 T^2 Q^{-1} \ll_{\varepsilon} P^{\varepsilon} Q^2T^8\]
for $N \leq Q^{3/2} T^3P^{\varepsilon}$, and this majorizes all preceding error terms.

The contribution of  $\mathcal{A}^{\textnormal{hol}}_c(n_1, n_2; \mathscr{L}^{\textnormal{hol}}\mathscr{K}_sh_{  T})$ can be bounded in same way using the analogous bound for $\mathscr{L}^{\textnormal{hol}}\mathscr{K}_sh_{  T}$ in Lemma \ref{hard-K}b).

Finally, for the contribution $\mathcal{A}^{\textnormal{Eis}, \ast}_c(n_1, n_2; \mathscr{L}^+\mathscr{K}_sh_{  T}) $, we   observe that after removing the trivial character, the analogously defined function
\[\mathcal{D}^{\textnormal{Eis}}_{(\chi, t), c}( {\tt z}_1,  {\tt  z}_2, {\tt w}_1, {\tt w}_2 )  = |L(1 + 2it, \chi^2)|^2 c\nu(c) \sum_{c_{\chi}^2 \mid M \mid c} \sum_{ n_1, n_2, m_1, m_2 } \frac{A(n_1, m_1)A(n_2, m_2) \rho_{\chi, M, c}(n) \overline{ \rho_{\chi, M, c}(m)}  }{(n_1m_1^2)^{{\tt z}_1  }(n_2  m_2^2)^{{\tt z}_2  } m_1^{2{\tt w}_1} m_2^{2{\tt w}_2}}\]
is pole-free in $\Re {\tt z}_1, \Re {\tt z}_2, \Re ({\tt z}_1 + {\tt w}_1), \Re ({\tt z}_2+ {\tt w}_2) \geq 1/2 + \varepsilon$ since $\chi$ is primitive of conductor $> 1$, and it can be approximated by
$L({\tt z}_1 + it, \chi)^3L({\tt z}_1 - it, \overline{\chi})^3L({\tt z}_2 + it, \chi)^3L({\tt z}_2 - it, \overline{\chi})^3$  in this region up to a holomorphic factor bounded by $O_{\varepsilon}(c^{\varepsilon})$. Here we can even afford to apply the convexity bound for all twelve Dirichlet $L$-functions. 
 The quantity corresponding to \eqref{maass-final} is then
\[P^{\varepsilon} Q \Big(\frac{NT}{Q}\Big)^{3/2}   \sum_{ c} \int_{ c(1+|t|)^2 \leq P^{\varepsilon} NT/Q}   \frac{\#\{\chi : c_{\chi}^2 \mid c\} T^2}{c(1 + |t|^2)} \, dt \ll_{\varepsilon} P^{\varepsilon} Q \Big(\frac{NT}{Q}\Big)^{3/2} \ll_{\varepsilon} P^{\varepsilon} Q^{7/4}T^6.\]
 This completes the proof of Theorem \ref{thm2}. 
 
\section{Applications}

It is now an easy task to prove Corollary \ref{cor3} and Theorem \ref{thm4}. For both applications, we need the following auxiliary result.
 \begin{lemma}\label{lem1} Let $T \geq 1$, $N, q \in \mathbb{N}$, $(N, q) = 1$, then
\[\sum_{\psi \in \mathcal{B}^{\ast}(N)} \frac{\lambda_{\psi}(q)^2}{L(1, \mathrm{Ad}^2\psi)} e^{-(t_{\psi}/T)^2} \ll_{\varepsilon} (NTq)^{\varepsilon}(T^2N + q^{1/2}).\]
 \end{lemma}
 
This is a simple application of the Kuznetsov formula and Weil's bounds for Kloosterman sums, cf.\ e.g.\  \cite[Lemma 12]{BM} or its ancestor \cite[Lemma 2.4]{Mo}.

\subsection{Proof of Corollary \ref{cor3}} From \cite[Section 12.1]{BK1}, we quote
\begin{equation}\label{quote}
\sum_{\psi \in \mathcal{B}^{\ast}(q)} L(1/2, \psi)^5 e^{-t_{\psi}^2} \ll_{\varepsilon} q^{\varepsilon} \max_{|\tau| \leq (\log q)^2} \sum_{\ell \leq q^{1/2 + \varepsilon}} \frac{1}{\ell^{1/2}} \Bigl| \sum_{f\in \mathcal{B}^{\ast}(q)} \frac{L(1/2, f)^4}{L(1, \mathrm{Ad}^2 f)} \lambda_f(\ell)h_{\tau}(t_f) \Bigr| + q^{-10},
\end{equation}
where
\[h_{\tau}(t_f) = \frac{L_{\infty}(1/2 + \varepsilon + i\tau, f)}{L_{\infty}(1/2, f)} \frac{G_f(\varepsilon + i\tau)}{G_f(0)} e^{-  t_f^2} ( 1+ |t_f|)^{\varepsilon}\]
with 
\[G_f(s) = \prod_{j=0}^{1000} \prod_{\epsilon_1, \epsilon_2 \in \{\pm 1\}} \left( \frac{1}{2} + \epsilon_1 s + i \epsilon_2 t_f + j\right).\]
This is an application of a carefully designed approximate functional equation. 
Now the formula two displays below \cite[(11.4)]{BK1} together with \cite[Lemma 1]{BK1} show that for $q$ prime
\[\frac{\phi(q)}{q^2} \sum_{f\in \mathcal{B}^{\ast}(q)} \frac{L(1/2, f)^4}{L(1, \mathrm{Ad}^2 f)} \lambda_f(\ell)h_{\tau}(t_f) = \sum_{ab = \ell} \mathcal{M}^+_{q, a} (1/2, 1/2, \mathfrak{h}_{\tau}) \left(\frac{a}{b}\right)^{1/2} + O(\ell^{\theta + \varepsilon} q^{-1})\]
with $\mathfrak{h}_{\tau} = (h_{\tau}, 0)$ in the notation of \cite[(1.3), (1.7)]{BK1}. Here the error term also includes the oldforms of level 1. On the other hand,  \cite[(11.4)]{BK1} states
\[\sum_{a b =  \ell}  \mathcal{M}^+_{q, a}(1/2, 1/2,  \mathfrak{h}_{\tau})  \left(\frac{a}{b }\right)^{\frac{1}{2}} \ll_{\varepsilon} (\ell q)^{\varepsilon}\sum_{a b =  \ell}  \left(\frac{a}{b}\right)^{\frac{1}{2}} \Bigr(\frac{1}{a} + \frac{1}{q} + \sum_{\pm}\big |\mathcal{M}^{\pm}_{a, q}(1/2, 1/2;   \mathscr{T}_{1/2, 1/2}^{\pm}\mathfrak{h}_{\tau})\big|\Bigr)\]
uniformly in $|\tau| \leq (\log q)^2$, and the analysis of \cite[Section 11]{BK1} shows
\[\mathcal{M}^{\pm}_{a, q}(1/2, 1/2;   \mathscr{T}_{1/2, 1/2}^{\pm}\mathfrak{h}_{\tau}) \ll_{\varepsilon} (aq)^{\varepsilon} \Bigl(\frac{1}{q^{1/2}} + \frac{1}{a q^{1/2}} \sum_{a_0 \mid a} \sum_{f\in \mathcal{B}^{\ast}(a_0)}  \frac{|L(1/2, f)|^4}{L(1, \mathrm{Ad}^2 f)} \frac{ |\Lambda_f(q, 1/2)|}{(1+|t_f|)^{15}}\Bigr)\]
(again uniformly in $|\tau| \leq (\log q)^2$) where $\Lambda_f(q, 1/2) \defeq \lambda_f(q) - q^{-1/2}$ for $q$ prime. Combining these estimates, we obtain
\begin{displaymath}
\begin{split}
\frac{\phi(q)}{q^2} \sum_{f\in \mathcal{B}^{\ast}(q)} &\frac{L(1/2, f)^4}{L(1, \mathrm{Ad}^2 f)} \lambda_f(\ell)h_{\tau}(t_f)\\
& \ll_{\varepsilon} (\ell q)^{\varepsilon} \Bigl( \frac{1}{\ell^{1/2}} + \frac{\ell^{1/2}}{q^{1/2}} + \frac{1}{  (\ell q)^{1/2} }\sum_{a_0  \mid  \ell} \sum_{f\in \mathcal{B}^{\ast}(a_0)} \frac{|L(1/2, f)|^4}{L(1, \mathrm{Ad}^2 \psi)} \frac{(1 + |\lambda_f(q)|)  }{(1 + |t_f|)^{15}} \Bigr).
 \end{split}
 \end{displaymath}
Substituting back into \eqref{quote}, this yields
\begin{displaymath}
\begin{split}
\sum_{\psi \in \mathcal{B}^{\ast}(q)} L(1/2, \psi)^5 e^{-t_{\psi}^2}& \ll_{\varepsilon} q^{1+\varepsilon}   \sum_{\ell \leq q^{1/2 + \varepsilon}}\Bigl( \frac{1}{\ell}   + \frac{1}{q^{1/2}} + \frac{1}{\ell q^{1/2}}\sum_{a \mid  \ell} \sum_{f\in \mathcal{B}^{\ast}(a)}\frac{|L(1/2, f)|^4}{L(1, \mathrm{Ad}^2 \psi)} \frac{(1 + |\lambda_f(q)|)  }{(1 + |t_f|)^{15}}   \Bigr) \\
& \ll_{\varepsilon} q^{1+\varepsilon}   + q^{1/2+\varepsilon} \sum_{a \leq q^{1/2 +\varepsilon}} \frac{1}{a}  \sum_{f\in \mathcal{B}^{\ast}(a)}\frac{|L(1/2, f)|^4}{L(1, \mathrm{Ad}^2 \psi)} \frac{(1 + |\lambda_f(q)|)  }{(1 + |t_f|)^{15}}   \\
& \ll_{\varepsilon} q^{1+\varepsilon}   + q^{1/2+\varepsilon} \max_{A \leq q^{1/2+\varepsilon}} \frac{1}{A} \sum_{a \asymp A}   \sum_{f\in \mathcal{B}^{\ast}(a)} \frac{|L(1/2, f)|^4}{L(1, \mathrm{Ad}^2 \psi)} \frac{(1 + |\lambda_f(q)|)  }{(1 + |t_f|)^{15}} .
\end{split}
\end{displaymath}
So far this is essentially a re-statement of the analysis in \cite{BK1}, but now we insert an additional application of H\"older's inequality. In this way, we obtain
\begin{displaymath}
\begin{split}
&\sum_{\psi \in \mathcal{B}^{\ast}(q)} L(1/2, \psi)^5 e^{-t_{\psi}^2}\\
& \ll_{\varepsilon}  q^{1+\varepsilon}   + q^{1/2+\varepsilon} \max_{A \leq q^{1/2+\varepsilon}} \frac{1}{A}\Bigl( \sum_{a \asymp A}   \sum_{f\in \mathcal{B}^{\ast}(a)} \frac{|L(1/2, f)|^6}{L(1, \mathrm{Ad}^2 \psi)(1 + |t_f|)^{15}} \Bigr)^{2/3} \Bigl( \sum_{a \asymp A}   \sum_{f\in \mathcal{B}^{\ast}(a)} \frac{(1 + |\lambda_f(q)|)^3}{L(1, \mathrm{Ad}^2 \psi)(1 + |t_f|)^{15}} \Bigr)^{1/3}.
\end{split}
\end{displaymath}
By Theorem \ref{thm2} and Lemma \ref{lem1}, we obtain
\[\sum_{\psi \in \mathcal{B}^{\ast}(q)} L(1/2, \psi)^5 e^{-t_{\psi}^2} \ll_{\varepsilon}  q^{\varepsilon}\Big(q+ q^{1/2 } \max_{A \leq q^{1/2+\varepsilon}} \frac{1}{A} A^{4/3} (A^2 + q^{1/2})^{1/3} q^{{\theta/3}}\Big) \ll_{\varepsilon} q^{1+\theta/3 + \varepsilon}.\]

\subsection{Proof of Theorem \ref{thm4}} By a dyadic decomposition, we can replace the summation condition $m \leq M$ by $m \asymp M$.  Let us also assume without loss of generality that $\| \textbf{a} \|_{\infty} \leq 1$. In order to apply Theorem \ref{thm1}, we would like to bound $L(1/2, \chi)$ by a small integral over the imaginary axis. This can be done by a standard argument based on the functional equation and the residue theorem (which seems to have been first applied by Heath-Brown \cite[Lemma 3]{HB}) as follows. Fix $0 < \varepsilon < 1/10$ and suppose that $\chi$ is a primitive character modulo $q$. We have 
\begin{displaymath}
\begin{split}
 L(1/2, \chi)^4 &=   \int_{(\varepsilon)} L(1/2 + s, \chi)^4 \frac{e^{2s^2}}{s} \, \frac{ds}{2\pi i}  + \int_{(\varepsilon)} L(1/2 - s, \chi)^4 \frac{e^{2s^2}}{s} \, \frac{ds}{2\pi i}
 = \int_{(\varepsilon)} L(1/2 + s, \chi)^4 f(s) \, \frac{ds}{2\pi i},
\end{split}
\end{displaymath}
where
\[f(s)  = \Bigl( 1+ \frac{\Gamma(\frac{1}{2}(\frac{1}{2}+ s + \mathfrak{a}))^4}{\Gamma(\frac{1}{2}(\frac{1}{2}- s + \mathfrak{a}))^4} \left(\frac{q}{\pi}\right)^{4s}\Bigr) \frac{e^{2s^2}}{s}\]
with $\mathfrak{a} = 0$ if $\chi$ is even and $\mathfrak{a} = 1$ if $\chi$ is odd. 
Applying the same argument again, we have
\begin{displaymath}
\begin{split}
L(1/2 + s, \chi)^4 = \int_{(-\varepsilon)} L(1/2 + s+u, \chi)^4 g_s(u) \, \frac{ds}{2\pi i},
\end{split}
\end{displaymath}
 where
\[g_{s}(u) = - \Bigl( 1+ \frac{\Gamma(\frac{1}{2}(\frac{1}{2}+ s +u+ \mathfrak{a}))^4}{\Gamma(\frac{1}{2}(\frac{1}{2}+ s -u+ \mathfrak{a}))^4} \left(\frac{q}{\pi}\right)^{4u}\Bigr)\frac{e^{2u^2}}{u}.\]
Inserting and changing variables, we obtain 
\begin{displaymath}
\begin{split}
 L(1/2, \chi)^4  &= \int_{(0)} L(1/2 + v, \chi)^4 h(v) \, \frac{dv}{2\pi i },
   \end{split}
\end{displaymath}
where
\begin{displaymath}
\begin{split}
h(v) & =  \int_{(\varepsilon)} g_s(v-s) f(s) \, \frac{ds}{2\pi i} \ll e^{-|v|^2} q^{4\varepsilon} .\\
  \end{split}
\end{displaymath}
 Now choosing $F$ as in \eqref{defF}, we get
\[| L(1/2, \chi)|^4  \ll_{\varepsilon} q^{\varepsilon} \int_{(0)} L(1/2 + z, \chi)^2L(1/2 - z, \overline{\chi})^2 F(z) \, \frac{dz}{2\pi i }.\]
 Opening the square, we have
\begin{displaymath}
\begin{split}
&\sum_{\chi\hspace{-.25cm} \pmod{q}} \Big|\sum_{m \asymp M} a(m) \chi(m)\Big|^2 |L(1/2, \chi)|^4 \ll_{\varepsilon} M^2+ q^{\varepsilon}\sum_d \sum_{\substack{m_1, m_2 \asymp M/d\\ (m_1, m_2) = 1}} \big| \mathcal{T}_{m_1, m_2, q}(1/2, 1/2, 1/2)\big|
\end{split}
\end{displaymath}
with the notation  as in \eqref{defT} in the special case where $f$ is the standard Eisenstein series with $\theta = 0$. By Theorem \ref{thm1} and \eqref{P} we have 
\begin{displaymath}
\begin{split}
&\sum_{\chi\hspace{-.25cm} \pmod{q}} \Big|\sum_{m \asymp M} a(m) \chi(m)\Big|^2 |L(1/2, \chi)|^4\\
& \ll_{\varepsilon} M^2+ (Mq)^{1+\varepsilon} + q^{\varepsilon}\sum_d \sum_{\substack{m_1, m_2 \asymp M/d\\ (m_1, m_2) = 1}} \sum_{\ast
 \in \{\textnormal{Maa{\ss}}, \textnormal{hol}, \textnormal{Eis}\}}\big| \mathcal{M}^{\ast}_{m_1, m_2, q}(1/2, 1/2, 1/2)\big|.
\end{split}
\end{displaymath} 
We only deal with the Maa{\ss} case; the other two cases are similar but easier. By \eqref{defMaass} and \eqref{boundTheta}, we have
\begin{displaymath}
\begin{split}
&\sum_d \sum_{\substack{m_1, m_2 \asymp M/d\\ (m_1, m_2) = 1}} \big| \mathcal{M}^{\textnormal{Maa{\ss}}}_{m_1, m_2, q}(1/2, 1/2, 1/2)\big|\\
&\ll_{\varepsilon} (qM)^{\varepsilon} \sum_d \sum_{\substack{m_1, m_2 \asymp M/d\\ (m_1, m_2) = 1}} \sum_{N \mid m_1m_2}\sum_{\psi \in \mathcal{B}^{\ast}(N)} \frac{q^{1/2}}{N^{1/2}}\frac{(1+|\lambda_{\psi}(q)|)}{(1+|t_{\psi}|)^{30}} \frac{L(1/2, \psi)^3}{L(1, \mathrm{Ad}^2\psi)}.
\end{split}
\end{displaymath}
We drop the condition $(m_1, m_2) = 1$ and write $ m_1m_2 = m = NK$, obtaining  by the standard divisor bound that the previous display is bounded by 
\begin{displaymath}
\begin{split}
& \ll_{\varepsilon}  (qM)^{\varepsilon} \sum_d \sum_{ NK \asymp M^2/d^2} \frac{q^{1/2}}{N^{1/2}}\sum_{\psi \in \mathcal{B}^{\ast}(N)}\frac{(1+|\lambda_{\psi}(q)|)}{(1+|t_{\psi}|)^{30}} \frac{L(1/2, \psi)^3}{L(1, \mathrm{Ad}^2\psi)}\\
&\ll_{\varepsilon}   (qM)^{\varepsilon} \sum_d \sum_{ N \ll  M^2/d^2} \frac{q^{1/2} M^2}{d^2N^{3/2}}\sum_{\psi \in \mathcal{B}^{\ast}(N)}\frac{(1+|\lambda_{\psi}(q)|)}{(1+|t_{\psi}|)^{30}} \frac{L(1/2, \psi)^3}{L(1, \mathrm{Ad}^2\psi)}\\
& \ll_{\varepsilon}   (  q^{1/2}M^2)^{1+\varepsilon} \max_{\mathcal{N} \ll M^2}  \frac{1}{\mathcal{N}^{3/2}} \sum_{ N \asymp \mathcal{N}}  \sum_{\psi \in \mathcal{B}^{\ast}(N)}\frac{(1+|\lambda_{\psi}(q)|)}{(1+|t_{\psi}|)^{30}} \frac{L(1/2, \psi)^3}{L(1, \mathrm{Ad}^2\psi)}.\\
\end{split}
\end{displaymath}
By the Cauchy--Schwarz inequality, Theorem \ref{thm2}, and Lemma \ref{lem1}, this is 
\[\ll_{\varepsilon}    (q^{1/2}M^2)^{1+\varepsilon} \max_{\mathcal{N} \ll M^2}  \frac{1}{\mathcal{N}^{3/2}} (\mathcal{N} + q^{1/4}\mathcal{N}^{1/2}) \mathcal{N} \ll_{\varepsilon} (q^{1/2}M^2)^{1+\varepsilon}(M + q^{1/4}).\]
For $M \leq q^{1/4}$, we obtain altogether
\[\sum_{\chi\hspace{-.25cm} \pmod{q}} \Big|\sum_{m \asymp M} a(m) \chi(m)\Big|^2 |L(1/2, \chi)|^4 \ll_{\varepsilon} (Mq)^{1+\varepsilon},\]
as desired.

\end{document}